\documentclass[12pt]{article}
\usepackage[]{amsmath,amssymb}
\usepackage{amscd}
\usepackage{latexsym}
\usepackage{cite}
\usepackage{amsthm}

\newtheorem{definition}{Definition}[section]
\newtheorem{theorem}[definition]{Theorem}
\newtheorem{lemma}[definition]{Lemma}
\newtheorem{corollary}[definition]{Corollary}

\newtheorem{example}[definition]{Example}
\newtheorem{conjecture}[definition]{Conjecture}
\newtheorem{problem}[definition]{Problem}
\newtheorem{note}[definition]{Note}

\newtheorem{proposition}[definition]{Proposition}

\begin{document}
\title{\bf 
The alternating central extension
\\
for 
the
positive part of
$U_q(\widehat{\mathfrak{sl}}_2)$ }
\author{
Paul Terwilliger 
}
\date{}

\maketitle
\begin{abstract}
This paper is about the positive part $U^+_q$ of 
the quantum group $U_q(\widehat{\mathfrak{sl}}_2)$.
The algebra $U^+_q$ has a presentation with
two generators $A,B$ that satisfy the cubic $q$-Serre relations.
Recently we introduced a type of element in $U^+_q$, said to be alternating. 
Each alternating element commutes with exactly one of
$A$, $B$, $qBA-q^{-1}AB$, $qAB-q^{-1}BA$; this gives four types of
alternating elements. There are infinitely many alternating elements of
each type, and these mutually commute.
In the present paper we use the alternating elements
to obtain a central extension $\mathcal U^+_q$ of
$U^+_q$. We define $\mathcal U^+_q$ by
generators and relations. These generators, said to be
alternating, are in bijection with the alternating elements of
$U^+_q$. 
We display a surjective algebra homomorphism
$\mathcal U^+_q \to U^+_q$ that sends
each alternating generator of 
$\mathcal U^+_q$ to the corresponding alternating element in $U^+_q$.
We adjust this homomorphism to obtain an algebra isomorphism
$\mathcal U_q^+ \to U^+_q \otimes  
\mathbb F \lbrack z_1, z_2,\ldots\rbrack$
where $\mathbb F$ is the ground field and
$\lbrace z_n\rbrace_{n=1}^\infty$ are mutually commuting indeterminates.
We show that the alternating generators 
form a PBW basis for 
$\mathcal U_q^+$.
We discuss how $\mathcal U^+_q$ is related
to the work of
Baseilhac, Koizumi, Shigechi concerning the $q$-Onsager
algebra and integrable lattice models.
\bigskip


 \end{abstract}
\section{Introduction}
The $q$-Onsager algebra $\mathcal O_q$
is often used
to investigate integrable lattice models
\cite{bas2,
bas1,
basXXZ,
basBel,
BK05,
bas4,
basKoi,
basnc}.
 In \cite{BK05} Baseilhac and Koizumi introduced a current algebra
$\mathcal A_q$ for $\mathcal O_q$, in order to solve boundary
integrable systems with hidden symmetries. In
\cite[Definition~3.1]{basnc} Baseilhac and Shigechi gave
a presentation of $\mathcal A_q$ by generators and relations.
The presentation is a bit complicated, and
the precise relationship between
 $\mathcal A_q$ and 
$\mathcal O_q$ is presently unknown. However see
 \cite[Conjectures~1,~2]{basBel} and
 \cite[Conjectures~4.5,~4.6]{z2z2z2}.
Hoping to shed light on 
the above relationship,
in the present paper we consider a limiting case
in which the technical details are less complicated.
Following \cite[Section~1]{alternating} we replace $\mathcal O_q$ by 
the positive part $U^+_q$  of the quantum group 
$U_q(\widehat{\mathfrak{sl}}_2)$. We introduce
an algebra $\mathcal U^+_q$ 
that is related to $U^+_q$ in roughly the same way that
$\mathcal A_q$ is related to $\mathcal O_q$. We describe
in detail how 
$\mathcal U^+_q$ 
 is related to $U^+_q$.
We will summarize our results after a few comments.
\medskip

\noindent We now give some
background information about $U^+_q$.
The algebra $U^+_q$ has a presentation with
two generators $A,B$ that satisfy the cubic $q$-Serre relations;
see Definition 
\ref{def:nposp} below.
In \cite{alternating}
we introduced a type of element in $U^+_q$, said to be alternating. 
As we showed in
\cite[Lemma 5.11]{alternating},
each alternating element commutes with exactly one of
$A$, $B$, $qBA-q^{-1}AB$, $qAB-q^{-1}BA$. This gives four types of
alternating elements,
denoted
\begin{align*}
\lbrace  W_{-k}\rbrace_{k\in \mathbb N}, \quad
\lbrace  W_{k+1}\rbrace_{k\in \mathbb N}, \quad
\lbrace  G_{k+1}\rbrace_{k\in \mathbb N}, \quad
\lbrace {\tilde G}_{k+1}\rbrace_{k\in \mathbb N}.
\end{align*}
By 
\cite[Lemma 5.11]{alternating} the alternating elements of
each type mutually commute.
\medskip

\noindent 
The alternating elements arise naturally in the following way.
Start with the free algebra $\mathbb V$ on two generators $x,y$.
The standard (linear) basis for
$\mathbb V$ consists of the words in $x,y$.
In
\cite{rosso1, rosso} M. Rosso introduced
an algebra structure on $\mathbb V$, called a
$q$-shuffle algebra.
For $u,v\in \lbrace x,y\rbrace$ their
$q$-shuffle product is
$u\star v = uv+q^{\langle u,v\rangle }vu$, where
$\langle u,v\rangle =2$
(resp. $\langle u,v\rangle =-2$)
if $u=v$ (resp.
 $u\not=v$).
Rosso gave an injective algebra homomorphism $\natural$
from $U^+_q$ into the $q$-shuffle algebra
${\mathbb V}$, that sends $A\mapsto x$ and $B\mapsto y$.
By
\cite[Definition~5.2]{alternating}
the map $\natural$ sends
\begin{align*}
&W_0 \mapsto x, \qquad W_{-1} \mapsto xyx, \qquad W_{-2} \mapsto xyxyx, \qquad \ldots
\\
&W_1 \mapsto y, \qquad W_{2} \mapsto yxy, \qquad W_{3} \mapsto yxyxy, \qquad \ldots
\\
&G_{1} \mapsto yx, \qquad G_{2} \mapsto yxyx,  \qquad G_3 \mapsto yxyxyx, \qquad \ldots
\\
&\tilde G_{1} \mapsto xy, \qquad \tilde G_{2} \mapsto 
xyxy,\qquad \tilde G_3 \mapsto xyxyxy, \qquad \ldots
\end{align*}
In \cite{alternating} we used $\natural$
to obtain many relations involving the 
 alternating elements; see
Lemmas \ref{lem:nrel1},
\ref{lem:nrel2} below.
These relations resemble  the defining relations for
$\mathcal A_q$ found in 
\cite[Definition~3.1]{basnc}.
We will say more
about
Lemmas \ref{lem:nrel1},
\ref{lem:nrel2} shortly.
In \cite[Section~10]{alternating} 
we used the alternating elements to obtain some
PBW bases for $U^+_q$.
For instance, in
\cite[Theorem~10.1]{alternating} 
we showed that the elements in order
\begin{align*}
\lbrace  W_{-k}\rbrace_{k \in \mathbb N}, \quad 
\lbrace {\tilde G}_{k+1}\rbrace_{k \in \mathbb N}, \quad
\lbrace  W_{k+1}\rbrace_{k\in \mathbb N}
\end{align*}
 give a PBW basis for $U^+_q$, said to be alternating
\cite[Definition~10.3]{alternating}.
\medskip

\noindent We now summarize the main results of the present paper.
We define
an algebra  $\mathcal U^+_q$  by 
generators
\begin{align}
\lbrace \mathcal W_{-k}\rbrace_{k \in \mathbb N}, \quad 
\lbrace  \mathcal W_{k+1}\rbrace_{k \in \mathbb N}, \quad
\lbrace  \mathcal G_{k+1}\rbrace_{k\in \mathbb N}, \quad
\lbrace \mathcal {\tilde G}_{k+1}\rbrace_{k \in \mathbb N}
\label{eq:AG}
\end{align}
and the relations in Lemmas \ref{lem:nrel1}, \ref{lem:nrel2}. 
The generators 
(\ref{eq:AG}) are called alternating.
By construction there 
 exists a surjective algebra homomorphism
$ \mathcal U^+_q \to U^+_q$ that sends
\begin{align*}
\mathcal W_{-k} \mapsto W_{-k},\quad \qquad
\mathcal W_{k+1} \mapsto W_{k+1},\quad \qquad
\mathcal G_{k} \mapsto G_{k},\quad \qquad
\mathcal {\tilde G}_{k} \mapsto \tilde G_{k}
\end{align*}
for $k \in \mathbb N$.
As we will see, this map is not injective.
Denote the ground field by $\mathbb F$ and
let $\lbrace z_n\rbrace_{n=1}^\infty$
denote mutually commuting indeterminates.
Let
$\mathbb F\lbrack z_1, z_2,\ldots\rbrack$
denote the algebra consisting of the polynomials in
$z_1, z_2,\ldots $ that have all coefficients in 
$\mathbb F$. 
For notational convenience define $z_0=1$.
We display an algebra isomorphism
$\varphi:
\mathcal U^+_q \to U^+_q \otimes 
\mathbb F \lbrack z_1, z_2,\ldots\rbrack$ that sends
\begin{align*}
\mathcal W_{-n} &\mapsto \sum_{k=0}^n W_{k-n} \otimes z_k,
\quad \qquad \qquad 
\mathcal W_{n+1} \mapsto \sum_{k=0}^n W_{n+1-k} \otimes z_k,
\\
\mathcal G_{n} &\mapsto \sum_{k=0}^n G_{n-k} \otimes z_k,
\quad \qquad \qquad
\mathcal {\tilde G}_{n} \mapsto \sum_{k=0}^n \tilde G_{n-k} \otimes z_k
\end{align*}
for $n \in \mathbb N$. In particular $\varphi$ sends
\begin{align*}
\mathcal W_0 \mapsto W_0 \otimes 1,
\qquad \qquad 
\mathcal W_1 \mapsto W_1 \otimes 1.
\end{align*}
We use $\varphi$ to obtain the following results.
Let $\mathcal Z$ denote the center of $\mathcal U^+_q$.
We show that $\mathcal Z$
is generated by $\lbrace Z^\vee_n\rbrace_{n=1}^\infty$,
where
\begin{align*}
 Z^\vee_n= \sum_{k=0}^n \mathcal G_k \mathcal {\tilde G}_{n-k} q^{n-2k}
-  q \sum_{k=0}^{n-1} \mathcal W_{-k} \mathcal W_{n-k} q^{n-1-2k}.
\end{align*}
We show that for $n\geq 1$, $\varphi$ sends $Z^\vee_n \mapsto 1 \otimes z^\vee_n$
where $z^\vee_n = \sum_{k=0}^n z_k z_{n-k} q^{n-2k}$.
We show that 
$\lbrace Z^\vee_n\rbrace_{n=1}^\infty$ are algebraically
independent.
Let $\langle \mathcal W_0, \mathcal W_1\rangle$
denote the subalgebra of 
 $\mathcal U^+_q$  generated by
$\mathcal W_0, \mathcal W_1$.
We show that the algebra $\langle \mathcal W_0, \mathcal W_1\rangle$
is isomorphic to 
$U^+_q$. We show that the multiplication map
\begin{align*}
\langle \mathcal W_0,\mathcal W_1\rangle \otimes \mathcal Z &\to
	       \mathcal U^+_q 
	       \\
w \otimes z &\mapsto      wz            
\end{align*}
is an algebra isomorphism.
We show that the alternating generators in order
\begin{align*}
\lbrace \mathcal W_{-k} \rbrace_{k \in \mathbb N}, \qquad 
\lbrace \mathcal G_{k+1} \rbrace_{k\in \mathbb N}, \qquad  
\lbrace \mathcal {\tilde G}_{k+1} \rbrace_{k\in \mathbb N}, \qquad  
\lbrace \mathcal W_{k+1} \rbrace_{k\in \mathbb N}
\end{align*}
give a PBW basis for $\mathcal U^+_q$.
Motivated by the above results, 
near the end of the paper we give some 
conjectures concerning $\mathcal A_q$ and $\mathcal O_q$.
\medskip

\noindent The paper is organized as follows.
In Section 2 we give some background information about $U^+_q$.
In Section 3 we introduce the algebra $\mathcal U^+_q$ and
describe its basic properties.
In Section 4 we obtain some results about
the polynomial algebra $\mathbb F\lbrack z_1, z_2,\ldots\rbrack$
that will be used in later sections.
In Section 5 we show that the map $\varphi$ is an algebra isomorphism.
In Section 6 we describe the center of $\mathcal U^+_q$ and also
the subalgebra of $\mathcal U^+_q$ generated by $\mathcal W_0, \mathcal W_1$.
In Section 7 we describe several ideals of $\mathcal U^+_q$, and
in Section 8 we describe some symmetries of $\mathcal U^+_q$.
In Section 9 we describe a grading of $\mathcal U^+_q$, that gets
used in Section 10 to establish a PBW basis for $\mathcal U^+_q$.
In Section 11 we give some conjectures concerning
$\mathcal A_q$ and $\mathcal O_q$.
Section 12 contains an acknowledgment, and Section 13 is an appendix
containing some technical details.

\section{The algebra $U^+_q$}

\noindent
We now begin our formal argument.
Recall the natural numbers $\mathbb N=\lbrace 0,1,2,\ldots\rbrace$
and integers $\mathbb Z = \lbrace 0,\pm 1, \pm 2,\ldots \rbrace$.
Let $\mathbb F$ denote a field.
We will be discussing vector spaces, tensor products,
and algebras.
Each vector space and tensor product discussed is over $\mathbb F$.
Each algebra discussed is associative, over $\mathbb F$,
and has a multiplicative identity.
A subalgebra has the same multiplicative identity as the parent algebra.
\medskip

\noindent Fix a nonzero $q \in \mathbb F$
that is not a root of unity.
Recall the notation
\begin{eqnarray*}
\lbrack n\rbrack_q = \frac{q^n-q^{-n}}{q-q^{-1}}
\qquad \qquad n \in \mathbb Z.
\end{eqnarray*}
\noindent For elements $X, Y$ in any algebra, define their
commutator and $q$-commutator by 
\begin{align*}
\lbrack X, Y \rbrack = XY-YX, \qquad \qquad
\lbrack X, Y \rbrack_q = q XY- q^{-1}YX.
\end{align*}
\noindent Note that 
\begin{align*}
\lbrack X, \lbrack X, \lbrack X, Y\rbrack_q \rbrack_{q^{-1}} \rbrack
= 
X^3Y-\lbrack 3\rbrack_q X^2YX+ 
\lbrack 3\rbrack_q XYX^2 -YX^3.
\end{align*}

\begin{definition}
\label{def:nposp}
\rm 
(See \cite[Corollary~3.2.6]{lusztig}.)
Define the algebra $U^+_q$ 
by generators $A,B$ and relations
\begin{eqnarray}
&&
\lbrack A, \lbrack A, \lbrack A, B\rbrack_q \rbrack_{q^{-1}} \rbrack=0,
\qquad \qquad 
\lbrack B, \lbrack B, \lbrack B, A\rbrack_q \rbrack_{q^{-1}}
\rbrack=0.
\label{eq:nqSerre1}
\end{eqnarray}
\noindent We call $U^+_q$ the {\it positive part of 
$U_q(\widehat{\mathfrak{sl}}_2)$}.
The relations (\ref{eq:nqSerre1})
are called the {\it $q$-Serre relations}.
\end{definition}

\noindent
We will be discussing automorphisms and antiautomorphisms.
For an algebra $\mathcal A$,
an {\it automorphism} of $\mathcal A$ is 
an algebra isomorphism $\mathcal A \to \mathcal A$.
The {\it opposite algebra} $\mathcal A^{\rm opp}$ consists
of the vector space $\mathcal A$ and multiplication map
$\mathcal A \times \mathcal A \to \mathcal A$, $(a, b) \mapsto ba$.
An {\it antiautomorphism} of $\mathcal A$ 
is an
algebra isomorphism $\mathcal A \to \mathcal A^{\rm opp}$.
\begin{lemma}
\label{lem:nAAut}
There exists an  automorphism $\sigma$ of
$U^+_q$ that swaps $A, B$. 
There exists an  antiautomorphism
$S$ of $U^+_q$ that fixes each of $A$, $B$.
\end{lemma}

\noindent We mention a grading for the algebra $U^+_q$.
The $q$-Serre relations are homogeneous in both
$A$ and $B$. Therefore
the algebra $U^+_q$ has an $(\mathbb N \times \mathbb N)$-grading 
for which $A$ and $B$ are homogeneous,
with degrees $(1,0)$ and $(0,1)$ respectively. 
The {\it trivial} homogeneous 
component of $U^+_q$ has degree $(0,0)$
and is equal to $\mathbb F 1$.
\medskip

\noindent The alternating elements of
$U^+_q$ were introduced in \cite{alternating}. There are four types, denoted 
\begin{align}
\label{eq:nWWGGn}
\lbrace  W_{-k}\rbrace_{k \in \mathbb N}, \quad 
\lbrace  W_{k+1}\rbrace_{k \in \mathbb N}, \quad
\lbrace  G_{k+1}\rbrace_{k\in \mathbb N}, \quad
\lbrace {\tilde G}_{k+1}\rbrace_{k \in \mathbb N}.
\end{align}
As we will review in Lemma
\ref{lem:nrecgen},
the above elements are obtained
from $A,B$ using a recursive procedure with initial conditions
$W_0 = A$ and $W_1 = B$. 
\medskip

\noindent In 
\cite{alternating}
we displayed many relations satisfied
by the alternating elements of $U^+_q$.
In the next three lemmas we list some of these relations.

\begin{lemma} 
\label{lem:nrel1}
{\rm (See \cite[Proposition~5.7]{alternating}.)}
For $k \in \mathbb N$
the following holds in $U^+_q$:
\begin{align}
&
 \lbrack  W_0,  W_{k+1}\rbrack= 
\lbrack  W_{-k},  W_{1}\rbrack=
(1-q^{-2})({\tilde G}_{k+1} -  G_{k+1}),
\label{eq:n3p1vv}
\\
&
\lbrack  W_0,  G_{k+1}\rbrack_q= 
\lbrack {{\tilde G}}_{k+1},  W_{0}\rbrack_q= 
 (q-q^{-1})W_{-k-1},
\label{eq:n3p2vv}
\\
&
\lbrack G_{k+1},  W_{1}\rbrack_q= 
\lbrack  W_{1}, { {\tilde G}}_{k+1}\rbrack_q= 
(q-q^{-1}) W_{k+2}.
\label{eq:n3p3vv}
\end{align}
\end{lemma}

\begin{lemma}
\label{lem:nrel2}
{\rm (See \cite[Proposition~5.9]{alternating}.)}
For $k, \ell \in \mathbb N$ the
following relations hold in $U^+_q$:
\begin{align}
&
\lbrack  W_{-k},  W_{-\ell}\rbrack=0,  \qquad 
\lbrack  W_{k+1},  W_{\ell+1}\rbrack= 0,
\label{eq:n3p4vv}
\\
&
\lbrack  W_{-k},  W_{\ell+1}\rbrack+
\lbrack W_{k+1},  W_{-\ell}\rbrack= 0,
\label{eq:n3p5vv}
\\
&
\lbrack  W_{-k},  G_{\ell+1}\rbrack+
\lbrack G_{k+1},  W_{-\ell}\rbrack= 0,
\label{eq:n3p6vv}
\\
&
\lbrack W_{-k},  {\tilde G}_{\ell+1}\rbrack+
\lbrack  {\tilde G}_{k+1},  W_{-\ell}\rbrack= 0,
\label{eq:n3p7vv}
\\
&
\lbrack  W_{k+1},  G_{\ell+1}\rbrack+
\lbrack   G_{k+1}, W_{\ell+1}\rbrack= 0,
\label{eq:n3p8vv}
\\
&
\lbrack  W_{k+1},  {\tilde G}_{\ell+1}\rbrack+
\lbrack  {\tilde G}_{k+1},  W_{\ell+1}\rbrack= 0,
\label{eq:n3p9vv}
\\
&
\lbrack  G_{k+1},  G_{\ell+1}\rbrack=0,
\qquad 
\lbrack {\tilde G}_{k+1},  {\tilde G}_{\ell+1}\rbrack= 0,
\label{eq:n3p10vv}
\\
&
\lbrack {\tilde G}_{k+1},  G_{\ell+1}\rbrack+
\lbrack  G_{k+1},  {\tilde G}_{\ell+1}\rbrack= 0.
\label{eq:n3p11v}
\end{align}
\end{lemma}

\begin{lemma}
\label{lem:rel3}
{\rm (See \cite[Proposition~5.10]{alternating}.)}
For $k,\ell \in \mathbb N$ the following
relations hold in $U^+_q$:
\begin{align}
&\lbrack W_{-k}, G_{\ell}\rbrack_q = 
\lbrack W_{-\ell}, G_{k}\rbrack_q,
\qquad \quad
\lbrack G_k, W_{\ell+1}\rbrack_q = 
\lbrack G_\ell, W_{k+1}\rbrack_q,
\label{eq:ngg1}
\\
&
\lbrack \tilde G_k, W_{-\ell}\rbrack_q = 
\lbrack \tilde G_\ell, W_{-k}\rbrack_q,
\qquad \quad 
\lbrack W_{\ell+1}, \tilde G_{k}\rbrack_q = 
\lbrack W_{k+1}, \tilde G_{\ell}\rbrack_q,
\label{eq:ngg2}
\\
&\lbrack G_{k}, \tilde G_{\ell+1}\rbrack -
\lbrack G_{\ell}, \tilde G_{k+1}\rbrack =
q\lbrack W_{-\ell}, W_{k+1}\rbrack_q-
q\lbrack W_{-k}, W_{\ell+1}\rbrack_q,
\label{eq:ngg3}
\\
&\lbrack \tilde G_{k},  G_{\ell+1}\rbrack -
\lbrack \tilde G_{\ell},  G_{k+1}\rbrack =
q \lbrack W_{\ell+1}, W_{-k}\rbrack_q-
q\lbrack W_{k+1}, W_{-\ell}\rbrack_q,
\label{eq:ngg4}
\\
&\lbrack G_{k+1}, \tilde G_{\ell+1}\rbrack_q -
\lbrack G_{\ell+1}, \tilde G_{k+1}\rbrack_q =
q\lbrack W_{-\ell}, W_{k+2}\rbrack-
q\lbrack W_{-k}, W_{\ell+2}\rbrack,
\label{eq:ngg5}
\\
&\lbrack \tilde G_{k+1},  G_{\ell+1}\rbrack_q -
\lbrack \tilde G_{\ell+1},  G_{k+1}\rbrack_q =
q \lbrack W_{\ell+1}, W_{-k-1}\rbrack-
q\lbrack W_{k+1}, W_{-\ell-1}\rbrack.
\label{eq:ngg6}
\end{align}
\end{lemma}

\begin{note}
\label{note:123}
\rm By \cite[Propositions~3.1,~3.2]{baspp}
the relations in Lemma \ref{lem:rel3}
are implied by the relations in
Lemmas
\ref{lem:nrel1},
\ref{lem:nrel2}. For this reason
we will give
Lemma \ref{lem:rel3}
less emphasis than
Lemmas
\ref{lem:nrel1},
\ref{lem:nrel2}.  
\end{note}

\begin{note}\rm The relations in
Lemmas
\ref{lem:nrel1},
\ref{lem:nrel2} resemble the defining relations
for $\mathcal A_q$ found in
\cite[Definition~3.1]{basnc}.
\end{note}

\noindent Consider
the four sequences in 
(\ref{eq:nWWGGn}). By
(\ref{eq:n3p4vv}),
(\ref{eq:n3p10vv}) the elements of each sequence mutually commute.
According to \cite[Lemma~5.11]{alternating},
\begin{enumerate}
\item[\rm (i)]
an alternating element commutes with $A$ if and only if
it is among
$\lbrace W_{-k}\rbrace_{k\in \mathbb N}$;
\item[\rm (ii)]
an alternating element commutes with $B$ if and only if
it is among
$\lbrace W_{k+1}\rbrace_{k\in \mathbb N}$;
\item[\rm (iii)]
an alternating element commutes with
$\lbrack B, A\rbrack_q$
if and only if it is among
$\lbrace G_{k+1}\rbrace_{k\in \mathbb N}$;
\item[\rm (iv)]
an alternating element commutes with 
$\lbrack A, B\rbrack_q$
if and only if it
is among
$\lbrace \tilde G_{k+1}\rbrace_{k\in \mathbb N}$.
\end{enumerate}

\noindent For notational convenience define $G_0=1$
and $\tilde G_0=1$.

\begin{lemma}
\label{prop:nGGWW}
{\rm (See \cite[Proposition~8.1]{alternating}.)}
For $n\geq 1$
the following hold in $U^+_q$:
\begin{align}
&
\sum_{k=0}^n  G_{k}  \tilde G_{n-k} q^{n-2k}
= q
\sum_{k=0}^{n-1} W_{-k} W_{n-k} q^{n-1-2k},
\label{eq:nGGWW1}
\\
&
\sum_{k=0}^n G_{k}  \tilde G_{n-k} q^{2k-n}
= q
\sum_{k=0}^{n-1} W_{n-k}  W_{-k} q^{n-1-2k},
\label{eq:nGGWW4}
\\
&
\sum_{k=0}^n  \tilde G_{k}   G_{n-k} q^{n-2k}
= q
\sum_{k=0}^{n-1} W_{n-k}  W_{-k} q^{2k+1-n},
\label{eq:nGGWW2}
\\
&
\sum_{k=0}^n \tilde G_{k}  G_{n-k} q^{2k-n}
= q
\sum_{k=0}^{n-1} W_{-k}  W_{n-k} q^{2k+1-n}.
\label{eq:nGGWW3}
\end{align}
\end{lemma}

\begin{lemma}
\label{lem:nrecgen}
{\rm (See \cite[Proposition~8.2]{alternating}.)}
Using the equations below, the alternating elements in
$U^+_q$
are recursively obtained from $A, B$ in the following order:
\begin{align*}
W_0, \quad W_1, \quad G_1, \quad \tilde G_1, \quad W_{-1}, \quad W_2, \quad
 G_2, \quad \tilde G_2, \quad W_{-2}, \quad W_3, \quad \ldots
\end{align*}
We have $W_0=A$ and $W_1=B$.
For $n\geq 1$,
\begin{align}
G_n &= \frac{q\sum_{k=0}^{n-1} W_{-k} W_{n-k} q^{n-1-2k}
-
\sum_{k=1}^{n-1} G_k  \tilde G_{n-k} q^{n-2k}}{q^n+q^{-n}}
+ 
\frac{W_n  W_0-W_0 W_n}{(1+q^{-2n})(1-q^{-2})},
\label{eq:nsolvG}
\\
\tilde G_n &= G_n + \frac{W_0 W_n-W_n W_0}{1-q^{-2}},
\label{eq:nsolvGt}
\\
W_{-n} &= \frac{q W_0 G_n - q^{-1} G_n  W_0}{q-q^{-1}},
\label{eq:nsolvWm}
\\
W_{n+1} &= \frac{q G_n W_1 - q^{-1} W_1  G_n}{q-q^{-1}}.
\label{eq:nsolvWp}
\end{align}
\end{lemma}

\begin{lemma}
\label{lem:nsigSact}
{\rm (See \cite[Proposition~5.3]{alternating}.)}
The maps $\sigma$, $S$  from
Lemma 
\ref{lem:nAAut}
act on the alternating
elements as follows. For $k \in \mathbb N$,
\begin{enumerate}
\item[\rm(i)]
the map $\sigma $ sends
\begin{align*}
W_{-k} \mapsto W_{k+1}, \qquad \quad
W_{k+1} \mapsto W_{-k}, \qquad \quad
G_{k} \mapsto \tilde G_{k}, \qquad \quad
\tilde G_{k} \mapsto G_{k};
\end{align*}
\item[\rm (ii)] the map $S$ sends
\begin{align*}
W_{-k} \mapsto W_{-k}, \qquad \quad
W_{k+1} \mapsto W_{k+1}, \qquad \quad
G_{k} \mapsto \tilde G_{k}, \qquad \quad
\tilde G_{k} \mapsto G_{k}.
\end{align*}
\end{enumerate}
\end{lemma}

\begin{lemma}
\label{lem:althom}
{\rm (See \cite[Section~5]{alternating}.)}
The alternating elements of
$U^+_q$ are homogeneous, with degrees shown below:
\bigskip

\centerline{
\begin{tabular}[t]{c|c}
{\rm alternating element} & {\rm degree} 
   \\  \hline
   $ W_{-k}$ & $(k+1,k)$
	 \\
   $ W_{k+1}$ & $(k,k+1)$
          \\
   $ G_{k}$ & $(k,k)$
          \\
   $ \tilde G_{k}$ & $(k,k)$
	 \end{tabular}
              }
              \medskip

\end{lemma}

\section{The algebra $\mathcal U^+_q$}

\noindent Motivated by 
Lemmas
\ref{lem:nrel1},
\ref{lem:nrel2}
and 
\cite[Definition~3.1]{basnc},
we now introduce
the algebra $\mathcal U^+_q$.

\begin{definition}
\label{def:calU}
\rm We define the algebra $\mathcal U^+_q$ by generators
\begin{align}
\label{eq:calWWGGn}
\lbrace \mathcal W_{-k}\rbrace_{k \in \mathbb N}, \quad 
\lbrace  \mathcal W_{k+1}\rbrace_{k \in \mathbb N}, \quad
\lbrace  \mathcal G_{k+1}\rbrace_{k\in \mathbb N}, \quad
\lbrace \mathcal {\tilde G}_{k+1}\rbrace_{k \in \mathbb N}
\end{align}
and relations
\begin{align}
&
 \lbrack \mathcal W_0,  \mathcal W_{k+1}\rbrack= 
\lbrack  \mathcal W_{-k}, \mathcal W_{1}\rbrack=
(1-q^{-2})(\mathcal {\tilde G}_{k+1} - \mathcal G_{k+1}),
\label{eq:n3p1vvC}
\\
&
\lbrack \mathcal W_0,  \mathcal G_{k+1}\rbrack_q= 
\lbrack {\mathcal {\tilde G}}_{k+1}, \mathcal W_{0}\rbrack_q= 
 (q-q^{-1})\mathcal W_{-k-1},
\label{eq:n3p2vvC}
\\
&
\lbrack \mathcal G_{k+1},  \mathcal W_{1}\rbrack_q= 
\lbrack  \mathcal W_{1}, { \mathcal {\tilde G}}_{k+1}\rbrack_q= 
(q-q^{-1}) \mathcal W_{k+2},
\label{eq:n3p3vvC}
\\
&
\lbrack \mathcal W_{-k},  \mathcal W_{-\ell}\rbrack=0,  \qquad 
\lbrack  \mathcal W_{k+1}, \mathcal  W_{\ell+1}\rbrack= 0,
\label{eq:n3p4vvC}
\\
&
\lbrack  \mathcal W_{-k},\mathcal  W_{\ell+1}\rbrack+
\lbrack \mathcal W_{k+1}, \mathcal W_{-\ell}\rbrack= 0,
\label{eq:n3p5vvC}
\\
&
\lbrack \mathcal  W_{-k},\mathcal  G_{\ell+1}\rbrack+
\lbrack \mathcal G_{k+1}, \mathcal W_{-\ell}\rbrack= 0,
\label{eq:n3p6vvC}
\\
&
\lbrack \mathcal W_{-k}, \mathcal {\tilde G}_{\ell+1}\rbrack+
\lbrack  \mathcal {\tilde G}_{k+1}, \mathcal  W_{-\ell}\rbrack= 0,
\label{eq:n3p7vvC}
\\
&
\lbrack  \mathcal W_{k+1}, \mathcal G_{\ell+1}\rbrack+
\lbrack  \mathcal G_{k+1},\mathcal W_{\ell+1}\rbrack= 0,
\label{eq:n3p8vvC}
\\
&
\lbrack  \mathcal W_{k+1},  \mathcal {\tilde G}_{\ell+1}\rbrack+
\lbrack  \mathcal {\tilde G}_{k+1},  \mathcal W_{\ell+1}\rbrack= 0,
\label{eq:n3p9vvC}
\\
&
\lbrack \mathcal G_{k+1},  \mathcal G_{\ell+1}\rbrack=0,
\qquad 
\lbrack \mathcal {\tilde G}_{k+1}, \mathcal {\tilde G}_{\ell+1}\rbrack= 0,
\label{eq:n3p10vvC}
\\
&
\lbrack \mathcal {\tilde G}_{k+1}, \mathcal G_{\ell+1}\rbrack+
\lbrack  \mathcal G_{k+1},  \mathcal {\tilde G}_{\ell+1}\rbrack= 0.
\label{eq:n3p11vC}
\end{align}
The generators (\ref{eq:calWWGGn}) are called {\it alternating}.
For notational convenience define
$ \mathcal G_{0}=1$ and
$\mathcal {\tilde G}_{0} = 1$.
\end{definition}

\begin{lemma} \label{calUrel3}
For $k,\ell \in \mathbb N$ the following
relations hold in $\mathcal U^+_q$:
\begin{align}
&\lbrack \mathcal W_{-k}, \mathcal G_{\ell}\rbrack_q = 
\lbrack \mathcal W_{-\ell}, \mathcal G_{k}\rbrack_q,
\qquad \quad
\lbrack \mathcal G_k, \mathcal W_{\ell+1}\rbrack_q = 
\lbrack \mathcal G_\ell, \mathcal W_{k+1}\rbrack_q,
\label{eq:calngg1}
\\
&
\lbrack \mathcal {\tilde G}_k, \mathcal W_{-\ell}\rbrack_q = 
\lbrack \mathcal {\tilde G}_\ell, \mathcal W_{-k}\rbrack_q,
\qquad \quad 
\lbrack \mathcal W_{\ell+1}, \mathcal {\tilde G}_{k}\rbrack_q = 
\lbrack \mathcal W_{k+1}, \mathcal {\tilde G}_{\ell}\rbrack_q,
\label{eq:calngg2}
\\
&\lbrack \mathcal G_{k}, \mathcal {\tilde G}_{\ell+1}\rbrack -
\lbrack \mathcal G_{\ell}, \mathcal {\tilde G}_{k+1}\rbrack =
q\lbrack \mathcal W_{-\ell}, \mathcal W_{k+1}\rbrack_q-
q\lbrack \mathcal W_{-k}, \mathcal W_{\ell+1}\rbrack_q,
\label{eq:calngg3}
\\
&\lbrack \mathcal  {\tilde G}_{k}, \mathcal  G_{\ell+1}\rbrack -
\lbrack \mathcal  {\tilde G}_{\ell}, \mathcal G_{k+1}\rbrack =
q \lbrack \mathcal  W_{\ell+1}, \mathcal W_{-k}\rbrack_q-
q\lbrack \mathcal W_{k+1}, \mathcal W_{-\ell}\rbrack_q,
\label{eq:calngg4}
\\
&\lbrack \mathcal G_{k+1}, \mathcal {\tilde G}_{\ell+1}\rbrack_q -
\lbrack \mathcal G_{\ell+1}, \mathcal {\tilde G}_{k+1}\rbrack_q =
q\lbrack \mathcal W_{-\ell}, \mathcal W_{k+2}\rbrack-
q\lbrack \mathcal W_{-k}, \mathcal W_{\ell+2}\rbrack,
\label{eq:calngg5}
\\
&\lbrack \mathcal {\tilde G}_{k+1},  \mathcal G_{\ell+1}\rbrack_q -
\lbrack \mathcal {\tilde G}_{\ell+1},  \mathcal G_{k+1}\rbrack_q =
q \lbrack \mathcal W_{\ell+1}, \mathcal W_{-k-1}\rbrack-
q\lbrack \mathcal W_{k+1}, \mathcal W_{-\ell-1}\rbrack.
\label{eq:calngg6}
\end{align}
\end{lemma}
\begin{proof} By Note
\ref{note:123}.
\end{proof}

\noindent The algebras $\mathcal U^+_q$ and $U^+_q$ are related
as follows.

\begin{lemma} 
\label{lem:gamma}
There exists an algebra homomorphism
$\gamma: \mathcal U^+_q \to U^+_q$ that sends
\begin{align*}
\mathcal W_{-n} \mapsto W_{-n},\quad \qquad
\mathcal W_{n+1} \mapsto W_{n+1},\quad \qquad
\mathcal G_{n} \mapsto G_{n},\quad \qquad
\mathcal {\tilde G}_{n} \mapsto \tilde G_{n}
\end{align*}
for $n \in \mathbb N$.
Moreover $\gamma$ is surjective.
\end{lemma} 
\begin{proof} 
By Definition
\ref{def:calU}.
\end{proof}

\noindent The kernel of $\gamma$ is described in Section 7.

\begin{definition}
\label{def:polyalg}
\rm
Let $\lbrace z_n\rbrace_{n=1}^\infty$
denote mutually commuting indeterminates. Let
$\mathbb F\lbrack z_1, z_2,\ldots\rbrack$
denote the algebra consisting of the polynomials in
$z_1, z_2,\ldots $ that have all coefficients in 
$\mathbb F$. 
For notational convenience define $z_0=1$.
\end{definition}

\begin{lemma}
\label{lem:eta}
There exists an algebra homomorphism
$\eta: \mathcal U^+_q\to 
 \mathbb F \lbrack z_1,z_2,
\ldots \rbrack$ that sends
\begin{align}
\mathcal W_{-n} \mapsto 0,\quad \qquad 
\mathcal W_{n+1} \mapsto 0,\quad \qquad 
\mathcal G_{n} \mapsto z_n, \quad \qquad 
\tilde {\mathcal  G}_{n} \mapsto z_n
\label{eq:etasends}
\end{align}
for $n \in \mathbb N$. 
Moreover $\eta$ is surjective.
\end{lemma}
\begin{proof} Use 
Definition
\ref{def:calU}.
\end{proof}

\noindent 
The kernel of $\eta$ is described in Section 7.

\medskip
\noindent We have indicated how $\mathcal U^+_q$ is related to
$U^+_q$ and
$\mathbb F\lbrack z_1,z_2,\ldots \rbrack$. Next we consider how
 $\mathcal U^+_q$ is related to
$U^+_q \otimes \mathbb F\lbrack z_1,z_2,\ldots \rbrack$. 
\medskip

\begin{lemma}
\label{lem:varphi}
There exists an algebra homomorphism $\varphi:
\mathcal U^+_q \to U^+_q \otimes 
\mathbb F \lbrack z_1, z_2,\ldots\rbrack$ that sends
\begin{align*}
\mathcal W_{-n} &\mapsto \sum_{k=0}^n W_{k-n} \otimes z_k,
\quad \qquad \qquad 
\mathcal W_{n+1} \mapsto \sum_{k=0}^n W_{n+1-k} \otimes z_k,
\\
\mathcal G_{n} &\mapsto \sum_{k=0}^n G_{n-k} \otimes z_k,
\quad \qquad \qquad
\mathcal {\tilde G}_{n} \mapsto \sum_{k=0}^n \tilde G_{n-k} \otimes z_k
\end{align*}
for $n \in \mathbb N$. In particular $\varphi$ sends
\begin{align}
\mathcal W_0 \mapsto W_0 \otimes 1,
\qquad \qquad 
\mathcal W_1 \mapsto W_1 \otimes 1.
\label{eq:vphiW0W1}
\end{align}
\end{lemma}
\begin{proof} Use Lemmas
\ref{lem:nrel1},
\ref{lem:nrel2} and 
Definition \ref{def:calU}.
\end{proof}

\noindent In Section 5 we  show that $\varphi$ is an isomorphism.
\medskip

\noindent Next we consider how $\gamma$ is related to $\varphi$.
There exists an algebra homomorphism
$\theta:
\mathbb F \lbrack z_1, z_2,\ldots \rbrack \to \mathbb F$
that sends $z_n \mapsto 0 $ for $n\geq 1$.
The map $\theta$ is surjective. Consequently 
 the vector space
 $\mathbb F \lbrack z_1, z_2,\ldots \rbrack$ is the direct sum of
$\mathbb F 1$ and the kernel of $\theta$.
This kernel is the ideal of 
 $\mathbb F \lbrack z_1, z_2,\ldots \rbrack$ generated by
$\lbrace z_n \rbrace_{n=1}^\infty$.

\begin{lemma} 
\label{lem:diag}
The following diagram commutes:

\begin{equation*}
{\begin{CD}
\mathcal U^+_q @>\varphi  >> U^+_q \otimes \mathbb F \lbrack z_1, z_2,\ldots
\rbrack
              \\
         @V \gamma VV                   @VV {\rm id} \otimes \theta V \\
         U^+_q @>>x \mapsto x\otimes 1 >
                                 U^+_q \otimes \mathbb F 
                        \end{CD}}  \qquad \quad \qquad 
			     \mbox{\rm id = identity map} 
\end{equation*}
\end{lemma}
\begin{proof} Chase each alternating generator of 
$\mathcal U^+_q$
around the diagram, using Lemmas
\ref{lem:gamma},
\ref{lem:varphi} and the definition of $\theta$.
\end{proof}

\noindent Next we consider how $\eta$ is related to $\varphi$.
Since $U^+_q$ is generated by $A,B$
and the $q$-Serre relations are homogeneous,
there exists an
algebra homomorphism $\vartheta : U^+_q \to \mathbb F$
that sends $A\mapsto 0$ and $B\mapsto 0$. 
The map $\vartheta$
is surjective, so $U^+_q$ is the direct sum of $\mathbb F 1$
and the kernel of $\vartheta$. The following are the same:
(i) the kernel of $\vartheta$;
(ii) the two-sided ideal
of $U^+_q$ generated by $A, B$;
(iii) the  sum of the nontrivial homogeneous components of $U^+_q$.
By Lemma
\ref{lem:althom} the map $\vartheta$ sends
\begin{align}
W_{-k}\mapsto 0,\qquad \quad
W_{k+1} \mapsto 0,
\qquad \quad 
G_{k+1}\mapsto 0,
\qquad \quad 
\tilde G_{k+1} \mapsto 0 
\label{eq:thetaSend}
\end{align}
for $k \in \mathbb N$.

\begin{lemma}
\label{lem:diagAgain}
The following diagram commutes:

\begin{equation*}
{\begin{CD}
\mathcal U^+_q @>\varphi  >> U^+_q \otimes \mathbb F \lbrack z_1, z_2,\ldots
\rbrack
              \\
         @V \eta VV                   @VV \vartheta \otimes {\rm id} V \\
         \mathbb F \lbrack z_1,z_2,\ldots \rbrack
	 @>>x \mapsto 1\otimes x >
                                 \mathbb F \otimes \mathbb F \lbrack
				 z_1,z_2,\ldots \rbrack
                        \end{CD}} 
\end{equation*}
\end{lemma}
\begin{proof} Chase each alternating generator of
$\mathcal U^+_q$ around the diagram,
using Lemmas
\ref{lem:eta},
\ref{lem:varphi} and 
(\ref{eq:thetaSend}).
\end{proof}

\noindent Next we describe some symmetries of $\mathcal U^+_q$.

\begin{lemma}
\label{lem:sym} 
There exists an automorphism $\sigma$ of
$\mathcal U^+_q$ that sends
\begin{align*}
\mathcal W_{-k} \mapsto \mathcal W_{k+1}, \qquad \quad
\mathcal W_{k+1} \mapsto \mathcal W_{-k}, \qquad \quad
\mathcal G_{k} \mapsto \mathcal{\tilde G}_{k}, \qquad \quad
\mathcal {\tilde G}_{k} \mapsto \mathcal G_{k}
\end{align*}
for $k \in \mathbb N$.
There exists an antiautomorphism $S$ of  
$\mathcal U^+_q$ that sends
\begin{align*}
\mathcal W_{-k} \mapsto \mathcal W_{-k}, \qquad \quad
\mathcal W_{k+1} \mapsto \mathcal W_{k+1}, \qquad \quad
\mathcal G_{k} \mapsto \mathcal {\tilde G}_{k}, \qquad \quad
\mathcal {\tilde G}_{k} \mapsto \mathcal G_{k}
\end{align*}
for $k \in \mathbb N$.
\end{lemma}
\begin{proof} Use
 Definition \ref{def:calU}.
\end{proof}

\noindent Next we introduce a grading for $\mathcal U^+_q$.
\begin{lemma}
\label{lem:calgrade}
The algebra $\mathcal U^+_q$ has an
$(\mathbb N \times\mathbb N)$-grading for which the
alternating generators are homogeneous, with degrees shown
below:
\bigskip

\centerline{
\begin{tabular}[t]{c|c}
{\rm alternating generator} & {\rm degree} 
   \\  \hline
   $ \mathcal W_{-k}$ & $(k+1,k)$
	 \\
   $ \mathcal W_{k+1}$ & $(k,k+1)$
          \\
   $ \mathcal G_{k}$ & $(k,k)$
          \\
   $ \mathcal {\tilde G}_{k}$ & $(k,k)$
	 \end{tabular}
              }
              \medskip

\end{lemma}
\begin{proof} The defining
relations for $\mathcal U^+_q$
are homogeneous with respect to the
above degree assignment.
\end{proof}

\section{The polynomial algebra $\mathbb F \lbrack z_1, z_2,\ldots\rbrack$}

\noindent Recall the algebra 
 $\mathbb F \lbrack z_1, z_2,\ldots\rbrack$ from
 Definition
\ref{def:polyalg}. 
In this section we obtain some results about
 $\mathbb F \lbrack z_1, z_2,\ldots\rbrack$
that will be used in later sections.

\begin{definition}
\rm
For $n \in \mathbb N$ define
\begin{align}
z^\vee_n = \sum_{k=0}^n z_k z_{n-k} q^{n-2k}.
\label{eq:zvee}
\end{align}
Note that $z^\vee_0=1$.
\end{definition}

\begin{example}\label{ex:zvee}
We have
\begin{align*}
z^\vee_1 &= (q+q^{-1})z_1,
\\
z^\vee_2 &= (q^2+q^{-2})z_2 + z^2_1,
\\
z^\vee_3 &= (q^3+q^{-3})z_3 + (q+q^{-1})z_1z_2.
\end{align*}
\end{example}

\begin{lemma}
\label{lem:zvpoly}
For $n\geq 1$ the element $z^\vee_n$ is a homogeneous polynomial
of total degree $n$ in $z_1, z_2, \ldots, z_n$,
where we view each $z_k$ as having degree $k$.
For this polynomial the coefficient of
$z_n$ is $q^n+q^{-n}$.
\end{lemma}
\begin{proof} By
(\ref{eq:zvee}).
\end{proof}

\noindent For $n\geq 1$, we now seek to express
$z_n$ as a polynomial in $z^\vee_1, z^\vee_2,\ldots, z^\vee_n$.
Towards this goal, we first express $z_n$ as a polynomial
in $z^\vee_n$ and
 $z_1, z_2,\ldots, z_{n-1}$.

\begin{lemma} 
\label{lem:zzvimpl}
For $n\geq 1$,
\begin{align}
z_n = \frac{z^\vee_n - \sum_{k=1}^{n-1} z_k z_{n-k} q^{n-2k}}{q^n+q^{-n}}.
\label{eq:znsolve}
\end{align}
\end{lemma}
\begin{proof} Solve 
(\ref{eq:zvee}) for $z_n$.
\end{proof}

\noindent For $n\geq 1$, we use 
Lemma 
\ref{lem:zzvimpl} and induction on $n$ to express
$z_n$
as a polynomial in $z^\vee_1, z^\vee_2,\ldots, z^\vee_n$.

\begin{example}
\label{ex:znsolve}
We have
\begin{align*}
z_1 &= \frac{z^\vee_1}{q+q^{-1}},
\\
z_2 &= \frac{(q+q^{-1})^2 z^\vee_2-(z^\vee_1)^2}{(q+q^{-1})^2(q^2+q^{-2})},
\\
z_3 &= \frac{(q+q^{-1})^2 (q^2+q^{-2})z^\vee_3-(q+q^{-1})^2z^\vee_1 z^\vee_2+
(z^\vee_1)^3}{(q+q^{-1})^2(q^2+q^{-2})(q^3+q^{-3})}.
\end{align*}
\end{example}

\begin{lemma}
\label{lem:zpoly}
For $n\geq 1$ the element $z_n$ is a homogeneous polynomial
of total degree $n$ in $z^\vee_1, z^\vee_2, \ldots, z^\vee_n$,
where we view each $z^\vee_k$ as having degree $k$.
For this polynomial the coefficient of
$z^\vee_n$ is $(q^n+q^{-n})^{-1}$.
\end{lemma}
\begin{proof} By
(\ref{eq:znsolve}) and induction on $n$.
\end{proof}

\begin{corollary}
\label{cor:homzzv}
The elements $\lbrace z^\vee_n\rbrace_{n=1}^\infty$
are algebraically independent and generate 
$\mathbb F \lbrack z_1, z_2,\ldots\rbrack$.
\end{corollary}
\begin{proof} 
The first assertion follows from
Lemma
\ref{lem:zvpoly} and since
$\lbrace z_n\rbrace_{n=1}^\infty$
are algebraically independent.
The second assertion follows from Lemma
\ref{lem:zpoly} and since
$\lbrace z_n\rbrace_{n=1}^\infty$ generate
$\mathbb F \lbrack z_1, z_2,\ldots\rbrack$.
\end{proof}

\begin{corollary}
\label{cor:autzz}
There exists an automorphism of the algebra
$\mathbb F \lbrack z_1, z_2,\ldots\rbrack$
that sends $z_n \mapsto z^\vee_n$ for $n\geq 1$.
\end{corollary}
\begin{proof} This is a reformulation of Corollary
\ref{cor:homzzv}.
\end{proof}

\section{The map $\varphi$ is an isomorphism}

\noindent Recall the map $\varphi$ from
Lemma
\ref{lem:varphi}.
In this section we show that $\varphi$
is an isomorphism.
\medskip

\noindent The following definition is motivated by
(\ref{eq:nGGWW1}).

\begin{definition} 
\label{def:Zvee}
\rm For $n \geq 1 $ define
\begin{align}
 Z^\vee_n= \sum_{k=0}^n \mathcal G_k \mathcal {\tilde G}_{n-k} q^{n-2k}
-  q \sum_{k=0}^{n-1} \mathcal W_{-k} \mathcal W_{n-k} q^{n-1-2k}.
\label{eq:ZnVee}
\end{align}
For notational convenience define $Z^\vee_0=1$.
\end{definition}

\noindent For any algebra $\mathcal A$, an element in
$\mathcal A$ is {\it central} whenever it commutes with
every element of $\mathcal A$. 
\begin{lemma}
\label{lem:Zcent}
\rm
For $n \in \mathbb N$ the element
$Z^\vee_n$ is central in $\mathcal U^+_q$.
\end{lemma}
\noindent The proof of Lemma 
\ref{lem:Zcent} is slightly technical, and contained in the
 Appendix.

\begin{note}\rm
The central elements 
(\ref{eq:ZnVee}) resemble the central elements for $\mathcal A_q$
given in 
\cite[Lemma~2.1]{basBel}.
\end{note}

\begin{lemma}
\label{lem:varphiZvee}
For $n \in \mathbb N$ the map $\varphi$ sends
$Z^\vee_n \mapsto 1 \otimes z^\vee_n$.
\end{lemma}
\begin{proof} Expand
$\varphi(Z^\vee_n)$ using
Lemma
\ref{lem:varphi} and
Definition
\ref{def:Zvee}. Evaluate the result using
(\ref{eq:nGGWW1}) and
(\ref{eq:zvee}).
\end{proof}

\begin{definition}
\label{def:calZ}
Let $\mathcal Z$ denote the subalgebra of
$\mathcal U^+_q$ generated by
$\lbrace Z^\vee_n\rbrace_{n=1}^\infty$.
\end{definition}

\noindent  For an algebra
$\mathcal A$, its central elements form a subalgebra called the
{\it center} of
$\mathcal A$.
By 
Lemma
\ref{lem:Zcent}
the subalgebra 
$\mathcal Z$ is contained in 
the center of $\mathcal U^+_q$.
In Section 6 we show that
$\mathcal Z$ is equal to
the center of $\mathcal U^+_q$.
\medskip

\noindent Next we introduce some elements
$\lbrace Z_n \rbrace_{n \in \mathbb N}$
in $\mathcal Z$, that are related to 
$\lbrace Z^\vee_n \rbrace_{n \in \mathbb N}$
in the same way that
 $\lbrace z_n \rbrace_{n \in \mathbb N}$
are related to 
$\lbrace z^\vee_n \rbrace_{n \in \mathbb N}$.
The elements 
$\lbrace Z_n \rbrace_{n \in \mathbb N}$
are defined recursively.

\begin{definition}
\label{def:Zn}
\rm
Define 
$Z_0=1$  and for $n\geq 1$,
\begin{align}
Z_n = \frac{Z^\vee_n - \sum_{k=1}^{n-1} Z_k Z_{n-k} q^{n-2k}}{q^n+q^{-n}}.
\label{eq:Znsol}
\end{align}
\end{definition}

\begin{lemma}
\label{lem:ZvZZ}
For  $n \in \mathbb N$,
\begin{align}
Z^\vee_n = \sum_{k=0}^n Z_k Z_{n-k} q^{n-2k}.
\label{eq:ZvZZ}
\end{align}
\end{lemma}
\begin{proof} 
Solve 
(\ref{eq:Znsol}) for $Z^\vee_n$.
\end{proof}

\begin{lemma} 
\label{lem:ZcZ}
The subalgebra
 $\mathcal Z$ from
Definition
\ref{def:calZ} is generated by 
$\lbrace Z_n \rbrace_{n=1}^\infty$.
\end{lemma}
\begin{proof} By Definition
\ref{def:calZ}
and Lemma
\ref{lem:ZvZZ}.
\end{proof}

\begin{lemma}
\label{lem:etaZ}
The map $\varphi$ sends
$Z_n \mapsto 1\otimes z_n$ for $n \in \mathbb N$.
\end{lemma}
\begin{proof} We use induction on $n$. The result holds for
$n=0$, since $Z_0=1$ and $z_0=1$.
Next assume $n\geq 1$. Using
Lemma
\ref{lem:varphiZvee}
and 
(\ref{eq:znsolve}),
(\ref{eq:Znsol})
 along with induction,
\begin{align*}
\varphi(Z_n) &= \frac{\varphi(Z^\vee_n) -
\sum_{k=1}^{n-1} \varphi(Z_k)\varphi(Z_{n-k}) q^{n-2k}}{q^n+q^{-n}}
\\
 &= \frac{1\otimes z^\vee_n -
\sum_{k=1}^{n-1}(1\otimes z_k)(1\otimes z_{n-k}) q^{n-2k}}{q^n+q^{-n}}
\\
 &= 1 \otimes \frac{z^\vee_n -
\sum_{k=1}^{n-1}z_k z_{n-k} q^{n-2k}}{q^n+q^{-n}}
\\
& =  1\otimes z_n.
\end{align*}
\end{proof}

\noindent We have a comment.
\begin{lemma}
\label{cor:19A}
The map $\varphi$ sends $\mathcal Z$ onto $\mathbb F \otimes
\mathbb F \lbrack z_1,z_2,\ldots \rbrack$.
\end{lemma}
\begin{proof} By Lemmas
\ref{lem:ZcZ},
\ref{lem:etaZ}.
\end{proof}

\noindent Next we show that the algebra $\mathcal U^+_q$
is generated by $\mathcal W_0, \mathcal W_1, \mathcal Z$.

\begin{lemma}
\label{lem:nrecgenCal}
Using the equations below, the alternating
generators of
$\mathcal U^+_q$
are recursively obtained from $\mathcal W_0, \mathcal W_1,
Z^\vee_1, Z^\vee_2, \ldots$ in the following order:
\begin{align*}
\mathcal W_0, \quad \mathcal W_1, \quad \mathcal G_1, 
\quad \mathcal {\tilde G}_1, \quad \mathcal  W_{-1}, 
\quad \mathcal W_2, \quad
\mathcal G_2, \quad \mathcal {\tilde G}_2, \quad 
\mathcal W_{-2}, \quad \mathcal W_3, \quad \ldots
\end{align*}
For $n\geq 1$,
\begin{align}
\mathcal G_n &= \frac{Z^\vee_n + q\sum_{k=0}^{n-1} 
\mathcal W_{-k} \mathcal W_{n-k} q^{n-1-2k}
-
\sum_{k=1}^{n-1}\mathcal G_k  \mathcal{\tilde G}_{n-k} q^{n-2k}}{q^n+q^{-n}}
+ 
\frac{\mathcal W_n \mathcal W_0-\mathcal W_0 \mathcal W_n}{(1+q^{-2n})(1-q^{-2})},
\label{eq:nsolvGCal}
\\
\mathcal {\tilde G}_n &= \mathcal G_n +
\frac{\mathcal W_0 \mathcal W_n-\mathcal W_n \mathcal W_0}{1-q^{-2}},
\label{eq:nsolvGtCal}
\\
\mathcal W_{-n} &= \frac{q  \mathcal W_0 \mathcal G_n
- q^{-1}\mathcal G_n \mathcal W_0}{q-q^{-1}},
\label{eq:nsolvWmCal}
\\
\mathcal W_{n+1} &= \frac{q \mathcal G_n \mathcal W_1 -
q^{-1} \mathcal W_1 \mathcal G_n}{q-q^{-1}}.
\label{eq:nsolvWpCal}
\end{align}
\end{lemma}
\begin{proof} 
Equation 
(\ref{eq:nsolvGtCal}) is from
(\ref{eq:n3p1vvC}).
To obtain (\ref{eq:nsolvGCal}), add
 $q^n$ times 
(\ref{eq:nsolvGtCal}) to
(\ref{eq:ZnVee}),
 and solve the resulting equation for $\mathcal G_n$.
Equations
(\ref{eq:nsolvWmCal}),
(\ref{eq:nsolvWpCal}) are from
(\ref{eq:n3p2vvC}),
(\ref{eq:n3p3vvC}).
\end{proof}

\begin{corollary}
\label{cor:calUgen}
The algebra $\mathcal U^+_q$ is generated by
 $\mathcal W_0, \mathcal W_1, \mathcal Z$.
\end{corollary}
\begin{proof} By Lemma
\ref{lem:nrecgenCal} and since
$\lbrace Z^\vee_n\rbrace_{n=1}^\infty$
 generate
$\mathcal Z$.
\end{proof}

\begin{note}\rm Lemma 
\ref{lem:nrecgenCal} and
Corollary \ref{cor:calUgen}
resemble the results for
$\mathcal A_q$ given in
\cite[Proposition~3.1]{basBel}
and \cite[Corollary~3.1]{basBel}.
\end{note}

\begin{lemma} 
\label{lem:qSerrecal}
 In $\mathcal U^+_q$ we have
\begin{align}
&
\lbrack \mathcal W_0, \lbrack \mathcal W_0, \lbrack \mathcal W_0,
\mathcal W_1\rbrack_q \rbrack_{q^{-1}} \rbrack=0,
\label{eq:WWqSerre1}
\\
&
\lbrack \mathcal W_1, \lbrack \mathcal W_1, \lbrack \mathcal W_1, 
\mathcal W_0\rbrack_q \rbrack_{q^{-1}}
\rbrack=0.
\label{eq:WWqSerre2}
\end{align}
\end{lemma}
\begin{proof} Consider 
(\ref{eq:WWqSerre1}). Setting $n=1$ in
(\ref{eq:nsolvGCal}),
(\ref{eq:nsolvWmCal}) we obtain
\begin{align}
\mathcal G_1 &= \frac{Z^\vee_1 + q \mathcal W_0 \mathcal W_1}{q+q^{-1}}
+ 
\frac{\lbrack \mathcal W_1, \mathcal W_0 \rbrack}{(1+q^{-2})(1-q^{-2})}
\nonumber
\\
&= \frac{Z^\vee_1}{q+q^{-1}}
- 
q \frac{\lbrack \mathcal W_0, \mathcal W_1 \rbrack_{q^{-1}}}{q^2-q^{-2}}
\label{G1form}
\end{align}
\noindent and
\begin{align}
\mathcal W_{-1} &=
\frac{\lbrack \mathcal W_0, \mathcal G_1 \rbrack_q}{q-q^{-1}}.
\label{eq:wmone}
\end{align}
\noindent 
The elements 
$\mathcal W_0, \mathcal W_{-1}$ commute by
(\ref{eq:n3p4vvC}), and $Z^\vee_1$ is central 
by Lemma
\ref{lem:Zcent}.
 By these comments and
(\ref{G1form}),
(\ref{eq:wmone}) 
we obtain
\begin{align*}
0 &= \lbrack \mathcal W_0, \mathcal W_{-1} \rbrack
\\
&= \frac{
 \lbrack \mathcal W_0, \lbrack \mathcal W_0,  \mathcal G_{1} 
 \rbrack_q
 \rbrack
 }{q-q^{-1}}
 \\
&= \frac{
 \lbrack \mathcal W_0, \lbrack \mathcal W_0,  \mathcal G_{1} 
 \rbrack
 \rbrack_q
 }{q-q^{-1}}
 \\
&= -q \frac{
 \lbrack \mathcal W_0, \lbrack \mathcal W_0,  \lbrack \mathcal W_0,
 \mathcal W_{1} 
 \rbrack_{q^{-1}}
 \rbrack
 \rbrack_q
 }{(q-q^{-1})(q^2-q^{-2})}
 \\
&= -q \frac{
 \lbrack \mathcal W_0, \lbrack \mathcal W_0,  \lbrack \mathcal W_0,
 \mathcal W_{1} 
 \rbrack_{q}
 \rbrack_{q^{-1}}
 \rbrack
 }{(q-q^{-1})(q^2-q^{-2})},
\end{align*}
which implies 
(\ref{eq:WWqSerre1}). 
The equation 
(\ref{eq:WWqSerre2}) is similarly obtained. 
\end{proof}

\begin{note}\rm
Lemma 
\ref{lem:qSerrecal} resembles
a result for $\mathcal A_q$ given in
\cite[Line~(3.7)]{basBel}.
\end{note}

\begin{lemma}
\label{lem:nphi}
There exists an algebra homomorphism $\phi:
U^+_q \otimes \mathbb F \lbrack z_1, z_2,\ldots \rbrack \to
\mathcal U^+_q$
that sends
\begin{align*}
W_0 \otimes 1 \mapsto \mathcal W_0,
\qquad \qquad
W_1 \otimes 1 \mapsto \mathcal W_1,
\qquad \qquad
1 \otimes z_n \mapsto Z_n, \qquad n \geq 1.
\end{align*}
\end{lemma}
\begin{proof} By Lemma
\ref{lem:qSerrecal} and since
$\lbrace Z_n \rbrace_{n=1}^\infty$ are central in
$\mathcal U^+_q$.
\end{proof}

\begin{theorem}
\label{thm:main1}
The maps $\varphi$, $\phi$ are inverses.
Moreover they are bijections.
\end{theorem}
\begin{proof} By
Lemma
\ref{lem:ZcZ}
and
Corollary
\ref{cor:calUgen}, the algebra $\mathcal U^+_q$ is generated by
$\mathcal W_0$,
$\mathcal W_1$,
$\lbrace Z_n \rbrace_{n=1}^\infty$.
Each of these generators is fixed by the composition
$\phi \circ \varphi$, in view of 
(\ref{eq:vphiW0W1}) and
Lemmas
\ref{lem:etaZ},
\ref{lem:nphi}. Therefore 
$\phi \circ \varphi$ is the
identity map on $\mathcal U^+_q$.
By Definition
\ref{def:nposp}
and the 
construction, the algebra
$U^+_q \otimes \mathbb F \lbrack z_1, z_2,\ldots \rbrack $
is generated by
$W_0\otimes 1$,
$W_1\otimes 1$,
$\lbrace 1 \otimes z_n \rbrace_{n=1}^\infty$.
Each of these generators is fixed by 
$\varphi \circ \phi$, in view of
(\ref{eq:vphiW0W1}) and
Lemmas
\ref{lem:etaZ},
\ref{lem:nphi}. Therefore 
$\varphi \circ \phi$ is the
identity map on
$U^+_q \otimes \mathbb F \lbrack z_1, z_2,\ldots \rbrack$.
By these comments the maps $\varphi,\phi$ are inverses, and
hence bijections.
\end{proof}

\section{Two subalgebras of $\mathcal U^+_q$}

\noindent In this section we describe 
two subalgebras of 
$\mathcal U^+_q$:  the
 center of $\mathcal U^+_q$ and 
the subalgebra 
generated by $\mathcal W_0$, $\mathcal W_1$.
\medskip

\noindent 
To obtain the center of $\mathcal U^+_q$ we will use
the following fact.

\begin{lemma} 
\label{lem:uc} {\rm (See \cite{yamane}.)} 
The center of  
 $U^+_q$ is $\mathbb F 1$.
\end{lemma}

\noindent Recall the subalgebra $\mathcal Z$ of
$\mathcal U^+_q$ described in
 Definition
\ref{def:calZ} and
Lemma
\ref{lem:ZcZ}.

\begin{proposition} 
\label{lem:UqC} 
The following {\rm (i)--(iv)} hold:
\begin{enumerate}
\item[\rm (i)] $\mathcal Z$ is the center of $\mathcal U^+_q$;
\item[\rm (ii)]  there exists an algebra isomorphism 
$\mathbb F \lbrack z_1, z_2,\ldots \rbrack \to \mathcal Z$
that sends $z_n \mapsto Z_n$ for $n \in \mathbb N$;
\item[\rm (iii)]  the above isomorphism sends
 $z^\vee_n \mapsto Z^\vee_n$ for $n \in \mathbb N$;
\item[\rm (iv)]  the inverse isomorphism is the
restriction of $\eta$ to $\mathcal Z$.
\end{enumerate}
\end{proposition}
\begin{proof} 
(i) By 
Lemma
\ref{lem:uc} 
and the construction, the center of 
$U^+_q \otimes \mathbb F \lbrack z_1, z_2,\ldots \rbrack $
is equal to  
$\mathbb F \otimes \mathbb F \lbrack z_1, z_2,\ldots \rbrack $.
Applying $\phi = \varphi^{-1}$ and Lemma
\ref{cor:19A}, we find that the center of
$\mathcal U^+_q$ is equal to $\mathcal Z$.
\\
(ii), (iv) Since $\lbrace Z_n\rbrace_{n=1}^\infty$ mutually
commute, there exists an algebra homomorphism
$\sharp : 
\mathbb F \lbrack z_1, z_2,\ldots \rbrack \to 
\mathcal U^+_q$ that sends
$z_n \mapsto Z_n$ for $n \geq 1$.
The map $\sharp$ has image $\mathcal Z$
by
Lemma
\ref{lem:ZcZ}. By Lemma
\ref{lem:diagAgain} and
Lemma
\ref{lem:etaZ},
the map $\eta$ sends
$Z_n \mapsto z_n$ for $n \geq 1$.
So the restriction of $\eta$ to
$\mathcal Z$ is the inverse of $\sharp$.
These maps are invertible and hence isomorphisms.
\\
\noindent (iii) Compare
(\ref{eq:zvee}),
(\ref{eq:ZvZZ}) and use (ii) above.
\end{proof}

\begin{corollary}
\label{cor:clarify}
The elements $\lbrace Z_n \rbrace_{n=1}^\infty$
are algebraically independent. Moreover the elements
 $\lbrace Z^\vee_n \rbrace_{n=1}^\infty$ are algebraically
 independent.
 \end{corollary}
\begin{proof} The first assertion follows from
Proposition
\ref{lem:UqC}(ii). The second assertion follows from
Corollary
\ref{cor:homzzv}
and
Proposition
\ref{lem:UqC}(iii).
\end{proof}

\noindent Let $\langle \mathcal W_0, \mathcal W_1\rangle$
denote the subalgebra of $\mathcal U^+_q$ generated by
$\mathcal W_0, \mathcal W_1$.

\begin{proposition}
\label{prop:gammainv}
The following {\rm (i), (ii)} hold:
\begin{enumerate}
\item[\rm (i)] there exists an algebra isomorphism
$U^+_q \to 
\langle \mathcal W_0, \mathcal W_1\rangle$
that sends $W_0 \mapsto \mathcal W_0$
and $W_1 \mapsto \mathcal W_1$;
\item[\rm (ii)] the inverse isomorphism is the restriction of
$\gamma$ to 
$\langle \mathcal W_0, \mathcal W_1\rangle$.
\end{enumerate}
\end{proposition}
\begin{proof} By Definition 
\ref{def:nposp} and
Lemma
\ref{lem:qSerrecal}, there exists an algebra homomorphism
$\flat: U^+_q \to \mathcal U^+_q$ that sends
$W_0 \mapsto \mathcal W_0$
and $W_1 \mapsto \mathcal W_1$.
By Lemma
\ref{lem:gamma} the map $\gamma$ sends
$\mathcal W_0 \mapsto  W_0$
and $\mathcal W_1 \mapsto  W_1$.
So the restriction of $\gamma$ to
$\langle \mathcal W_0, \mathcal W_1\rangle$ is
the inverse of $\flat$.
These maps are invertible and hence isomorphisms.
\end{proof}

\noindent Next we describe how the subalgebras $\mathcal Z$
and 
$\langle \mathcal W_0, \mathcal W_1\rangle$ are related.

\begin{proposition} 
\label{lem:UZ}
The multiplication map
\begin{align*}
\langle \mathcal W_0,\mathcal W_1\rangle \otimes \mathcal Z &\to
	       \mathcal U^+_q 
	       \\
w \otimes z &\mapsto      wz            
\end{align*}
is an algebra isomorphism.
\end{proposition}
\begin{proof} 
Let $m$ denote the above multiplication map.
The map $m$ is an algebra homomorphism since
$\mathcal Z$ is central in $\mathcal U^+_q$.
Let $\gamma_{\rm rest}$ denote the
restriction of $\gamma $ to
$\langle \mathcal W_0,\mathcal W_1\rangle $.
The map $\gamma_{\rm rest}: 
\langle \mathcal W_0,\mathcal W_1\rangle  \to U^+_q$
is an algebra isomorphism by Proposition
\ref{prop:gammainv}(ii).
Let $\eta_{\rm rest}$ denote the
 restriction of
$\eta$ to $\mathcal Z$.
The map $\eta_{\rm rest}: \mathcal Z \to 
             \mathbb F \lbrack z_1, z_2,\ldots \rbrack $
	     is an algebra isomorphism by Proposition
\ref{lem:UqC}(iv).
By these comments and
Theorem
\ref{thm:main1},
the composition
\begin{equation}
{\begin{CD}
\langle \mathcal W_0,\mathcal W_1\rangle \otimes \mathcal Z 
@>>\gamma_{\rm rest} \otimes \eta_{\rm rest}  >
             U^+_q \otimes \mathbb F \lbrack z_1, z_2,\ldots \rbrack 
         @>> \phi >  \mathcal U^+_q 
                        \end{CD}} 
\label{eq:mfact}
\end{equation}
is an algebra isomorphism.
The composition 
(\ref{eq:mfact})
is equal to 
$m$, since it agrees with $m$
on the generators $\mathcal W_0 \otimes 1$,
 $\mathcal W_1 \otimes 1$,
$\lbrace 1\otimes Z_n\rbrace_{n=1}^\infty$.
It follows that $m$ is an algebra isomorphism.
\end{proof}

\section{The kernels of $\gamma$ and $\eta$}

\noindent Recall the maps $\gamma$ and $\eta$ from
Section 3. In this section we describe their kernels.
We begin with $\gamma$.

\begin{proposition} The following are the same:
\begin{enumerate}
\item[\rm (i)] 
the kernel of $\gamma$;
\item[\rm (ii)]  the 2-sided ideal of $\mathcal U^+_q$
generated by $\lbrace Z_n \rbrace_{n=1}^\infty$;
\item[\rm (iii)]  the 2-sided ideal of $\mathcal U^+_q$
generated by $\lbrace Z^\vee_n \rbrace_{n=1}^\infty$.
\end{enumerate}
\end{proposition}
\begin{proof} (i), (ii) In the diagram of Lemma
\ref{lem:diag}, the two horizontal maps are bijections.
So the kernel of $\gamma$ is the $\varphi$-preimage of
the kernel of ${\rm id} \otimes \theta$.
The kernel of ${\rm id} \otimes \theta$ is obtained 
using the description of $\theta$
above Lemma \ref{lem:diag}.
\\
\noindent (iii) Use
(\ref{eq:Znsol}), 
(\ref{eq:ZvZZ}) and (ii) above.
\end{proof}

\begin{proposition} The vector space $\mathcal U^+_q$
is the direct sum of the following:
\begin{enumerate}
\item[\rm (i)]  the kernel of $\gamma$;
\item[\rm (ii)]  the subalgebra $\langle \mathcal W_0, \mathcal W_1\rangle$.
\end{enumerate}
\end{proposition}
\begin{proof} By Proposition
\ref{prop:gammainv}(ii) and linear algebra.
\end{proof}

\noindent We turn our attention to $\eta$.

\begin{proposition} The following are the same:
\begin{enumerate}
\item[\rm (i)] the kernel of $\eta$;
\item[\rm (ii)] the 2-sided ideal of $\mathcal U^+_q$
generated by $\mathcal W_0$, $\mathcal W_1$.
\end{enumerate}
\end{proposition}
\begin{proof} In the diagram of  Lemma
\ref{lem:diagAgain}, the two horizontal maps
are bijections. 
So the kernel of $\eta$ is the $\varphi$-preimage of
the kernel of $\vartheta\otimes {\rm id}$.
The kernel of $\vartheta\otimes {\rm id}$ is obtained 
using the description of $\vartheta$
above Lemma
\ref{lem:diagAgain}.
\end{proof}

\begin{proposition} 
The vector space $\mathcal U^+_q$ is the direct
sum of the following:
\begin{enumerate}
\item[\rm (i)] the center $\mathcal Z$ of $\mathcal U^+_q$;
\item[\rm (ii)] the kernel of $\eta$. 
\end{enumerate}
\end{proposition}
\begin{proof} 
By Proposition
\ref{lem:UqC}(iv) and linear algebra.
\end{proof}

\section{The automorphism $\sigma$ and antiautomorphism $S$}
\noindent
In Lemma \ref{lem:nAAut}
we gave
an automorphism $\sigma$ of $U^+_q$ 
and an 
 antiautomorphism $S$ of $U^+_q$.
In Lemma
\ref{lem:sym} we gave the
analogous maps for 
  $\mathcal U^+_q$.
In this section we describe how the
maps in Lemmas
\ref{lem:nAAut} and
\ref{lem:sym} are related.

\begin{lemma}
\label{lem:sigmadiag}
The following diagram commutes:
\begin{equation*}
{\begin{CD}
\mathcal U^+_q @>\varphi  >>
          U^+_q \otimes \mathbb F \lbrack z_1,z_2,\ldots \rbrack 
	      \\
         @V \sigma VV                   @VV \sigma \otimes {\rm id} V 
	 \\
\mathcal U^+_q @>>\varphi  >
          U^+_q \otimes \mathbb F \lbrack z_1,z_2,\ldots \rbrack 
			\end{CD}} 
   \end{equation*}
\end{lemma}
\begin{proof} Each map in the diagram is an algebra homomorphism.
To check that the diagram commutes, it suffices to chase each
alternating generator of $\mathcal U^+_q$
around the diagram. This chasing is routinely 
accomplished using Lemma
 \ref{lem:varphi} for the horizontal maps
 and Lemmas
\ref{lem:nsigSact},
\ref{lem:sym}  for the vertical maps.
\end{proof}

\begin{lemma}
\label{lem:Sdiag}
The following diagram commutes:
\begin{equation*}
{\begin{CD}
\mathcal U^+_q @>\varphi  >>
          U^+_q \otimes \mathbb F \lbrack z_1,z_2,\ldots \rbrack 
	      \\
         @V S VV                   @VV S \otimes {\rm id} V 
	 \\
\mathcal U^+_q @>>\varphi  >
          U^+_q \otimes \mathbb F \lbrack z_1,z_2,\ldots \rbrack 
			\end{CD}} 
   \end{equation*}
\end{lemma}
\begin{proof} Each horizontal (resp. vertical)
map in the diagram is an algebra homomorphism (resp. antiautomorphism).
To check that the diagram commutes, it suffices to chase each
alternating generator of $\mathcal U^+_q$
around the diagram. This chasing is routinely 
accomplished using
Lemma \ref{lem:varphi} for the horizontal maps
and 
Lemmas
\ref{lem:nsigSact},
\ref{lem:sym}  for the vertical maps.
\end{proof}

\begin{proposition} 
\label{prop:FIX}
Referring to the algebra $\mathcal U^+_q$,
\begin{enumerate}
\item[\rm (i)] 
the automorphism $\sigma$ fixes everything in $\mathcal Z$;
\item[\rm (ii)] 
the antiautomorphism $S$ fixes everything in $\mathcal Z$.
\end{enumerate}
\end{proposition}
\begin{proof}(i) By Lemma
\ref{lem:ZcZ}, it suffices to check that 
$\sigma(Z_n)=Z_n$ for $n \geq 1$.
This checking is routinely accomplished
by chasing $Z_n$ around the diagram in Lemma
\ref{lem:sigmadiag}, using the fact that $\varphi$ sends
$Z_n\mapsto 1\otimes z_n$
by Lemma
\ref{lem:etaZ},
and $\sigma \otimes {\rm id}$ sends $1 \otimes z_n \mapsto 1 \otimes z_n$ 
by construction.
\\
\noindent (ii) 
 By Lemma \ref{lem:ZcZ} and since 
 $\mathcal Z$ is commutative, it suffices to check that 
$S(Z_n)=Z_n$ for $n \geq 1$.
This checking is routinely accomplished
by chasing $Z_n$ around the diagram in Lemma
\ref{lem:Sdiag}, using the fact that $\varphi$ sends
$Z_n\mapsto 1\otimes z_n$
and $S \otimes {\rm id}$ sends $1 \otimes z_n \mapsto 1 \otimes z_n$ 
by construction.
\end{proof}

\begin{corollary} 
\label{cor:4same}
Referring to the algebra $\mathcal U^+_q$,
for $n\geq 1$ the element $Z^\vee_n$ is equal to each of the
following:
\begin{align}
&
\sum_{k=0}^n  \mathcal G_{k} \mathcal{ \tilde G}_{n-k} q^{n-2k}
- q
\sum_{k=0}^{n-1} \mathcal W_{-k} \mathcal W_{n-k} q^{n-1-2k},
\label{eq:SnGGWW1}
\\
&
\sum_{k=0}^n \mathcal G_{k}  \mathcal {\tilde G}_{n-k} q^{2k-n}
- q
\sum_{k=0}^{n-1} \mathcal W_{n-k} \mathcal W_{-k} q^{n-1-2k},
\label{eq:SnGGWW4}
\\
&
\sum_{k=0}^n \mathcal {\tilde G}_{k}   \mathcal G_{n-k} q^{n-2k}
- q
\sum_{k=0}^{n-1} \mathcal W_{n-k} \mathcal W_{-k} q^{2k+1-n},
\label{eq:SnGGWW2}
\\
&
\sum_{k=0}^n \mathcal {\tilde G}_{k}  \mathcal G_{n-k} q^{2k-n}
- q
\sum_{k=0}^{n-1} \mathcal W_{-k}  \mathcal W_{n-k} q^{2k+1-n}.
\label{eq:SnGGWW3}
\end{align}
\end{corollary}
\begin{proof} 
By Definition
\ref{def:Zvee}
the element $Z^\vee_n$ is equal to
the element (\ref{eq:SnGGWW1}).
By Proposition
\ref{prop:FIX} the element
$Z^\vee_n$ is fixed by
$\sigma$ and $S$.
By Lemma
\ref{lem:sym} the map
$\sigma$ sends the elements 
\begin{align*}
(\ref{eq:SnGGWW1}) \leftrightarrow (\ref{eq:SnGGWW2}),
\qquad \qquad
(\ref{eq:SnGGWW4}) \leftrightarrow (\ref{eq:SnGGWW3})
\end{align*}
and $S$ sends the elements
\begin{align*}
(\ref{eq:SnGGWW1}) \leftrightarrow (\ref{eq:SnGGWW4}),
\qquad \qquad
(\ref{eq:SnGGWW2}) \leftrightarrow (\ref{eq:SnGGWW3}).
\end{align*}
\noindent The result follows.
\end{proof}

\noindent 
It is illuminating to 
compare
 Lemma
\ref{prop:nGGWW} with Corollary
\ref{cor:4same}.

\section{The grading for $\mathcal U^+_q$}
In Lemma 
\ref{lem:calgrade}
we introduced an $(\mathbb N \times \mathbb N)$-grading for 
$\mathcal U^+_q$. In this section we compute the dimension
of each homogeneous component, and express the answer using
a generating function. Throughout this section let $\lambda$,
$\mu$ denote commuting indeterminates.
\medskip

\noindent 
We start with some comments about the 
$(\mathbb N \times \mathbb N)$-grading for $U^+_q$. This grading
was mentioned below
Lemma
\ref{lem:nAAut}
and described further in
Lemma
\ref{lem:althom}.

\begin{definition}
\label{def:gfuq}
\rm Define a generating function
\begin{align}
H(\lambda,\mu) = 
\prod_{n=1}^\infty 
\frac{1}{1-\lambda^{n} \mu^{n-1}}\,
\frac{1}{1-\lambda^{n} \mu^{n}}\,
\frac{1}{1-\lambda^{n-1} \mu^{n}}.
\label{eq:gfH}
\end{align}
\end{definition}
\begin{note}\rm 
In the product 
(\ref{eq:gfH}) 
we expand each factor using
 $(1-x)^{-1} = 1+ x + x^2 + \cdots$ to express
$H(\lambda, \mu)$ as a formal power series in $\lambda, \mu$.
We will do something similar for all the generating functions
encountered in this section.
\end{note}

\begin{lemma} 
\label{lem:ugrad}
{\rm (See \cite[Corollary~3.7]{alternating}.)}
For $(i,j) \in \mathbb N\times \mathbb N$ the following are the
same:
\begin{enumerate}
\item[\rm (i)] the dimension of the $(i,j)$-homogeneous component of
$U^+_q$;
\item[\rm (ii)] the coefficient of $\lambda^i \mu^j$
in $H(\lambda, \mu)$.
\end{enumerate}
\end{lemma}

\noindent Recall the algebra
$\mathbb F \lbrack z_1, z_2,\ldots \rbrack$
from Definition
\ref{def:polyalg}. Shortly we will endow this algebra
with an 
$(\mathbb N \times \mathbb N)$-grading.
The grading is motivated by the following result
concerning 
 the $(\mathbb N \times \mathbb N)$-grading
of $\mathcal U^+_q$.

\begin{lemma}
\label{lem:znn}
For $n\geq 1$ the elements $Z_n, Z^\vee_n$ are homogeneous
with degree $(n,n)$.
\end{lemma}
\begin{proof} For 
$Z^\vee_n$ use 
Lemma
\ref{lem:calgrade} and
Definition
\ref{def:Zvee}. For 
$Z_n$ use
(\ref{eq:Znsol}) 
and induction on $n$.
\end{proof}

\begin{definition} 
\label{lem:pgrading}
We endow the algebra
$\mathbb F \lbrack z_1, z_2,\ldots \rbrack$
with an
$(\mathbb N \times \mathbb N)$-grading such that
 $z_n$ is homogeneous with degree $(n,n)$ for $n\geq 1$.
\end{definition}

\begin{definition}
\label{def:genfuncZ}
\rm Define a generating function
\begin{align*}
Z(\lambda,\mu) = 
\prod_{n=1}^\infty 
\frac{1}{1-\lambda^{n} \mu^{n}}
\end{align*}
\end{definition}

\begin{lemma}
\label{lem:zgrad}
For $(i,j) \in \mathbb N\times \mathbb N$ the following are the
same:
\begin{enumerate}
\item[\rm (i)] the dimension of the $(i,j)$-homogeneous component of
$\mathbb F \lbrack z_1, z_2,\ldots \rbrack$;
\item[\rm (ii)] the coefficient of
$\lambda^i \mu^j$ in $Z(\lambda,\mu)$.
\end{enumerate}
\end{lemma}
\begin{proof} This is routinely checked.
\end{proof}

\noindent We have been discussing an 
$(\mathbb N\times \mathbb N)$-grading of $U^+_q$
and
an $(\mathbb N\times \mathbb N)$-grading of 
$\mathbb F \lbrack z_1, z_2,\ldots \rbrack$. We now combine
these gradings to get an 
 $(\mathbb N\times \mathbb N)$-grading of 
$U^+_q \otimes \mathbb F \lbrack z_1, z_2,\ldots \rbrack$.

\begin{lemma} 
\label{lem:fg}
The algebra
$U^+_q \otimes \mathbb F \lbrack z_1, z_2,\ldots \rbrack$
has a unique 
$(\mathbb N\times \mathbb N)$-grading with the following property:
for all homogeneous elements $f \in U^+_q$ and $g \in 
\mathbb F \lbrack z_1, z_2,\ldots \rbrack$, the element $f\otimes g \in 
U^+_q \otimes \mathbb F \lbrack z_1, z_2,\ldots \rbrack$
is homogeneous with 
${\rm deg}(f \otimes g) = {\rm deg}(f)+ {\rm deg}(g)$.
With respect to this grading each of
$W_0\otimes 1$,
$W_1\otimes 1$,
$\lbrace 1 \otimes z_n\rbrace_{n=1}^\infty$
is homogeneous with degree shown below:
\bigskip

\centerline{
\begin{tabular}[t]{c|c}
{\rm  element} & {\rm degree} 
   \\  \hline
   $ W_0 \otimes 1 $ & $(1,0)$
	 \\
   $ W_1 \otimes 1$ & $(0,1)$
          \\
   $ 1\otimes z_n$ & $(n,n)$
	 \end{tabular}
              }
              \medskip
\end{lemma}
\begin{proof} By construction.
\end{proof}

\begin{definition}
\label{def:calH}
\rm Define a generating function
\begin{align*}
\mathcal H(\lambda, \mu) &= H(\lambda, \mu)Z(\lambda,\mu)
\\
&= 
\prod_{n=1}^\infty 
\frac{1}{1-\lambda^{n} \mu^{n-1}}\,
\frac{1}{(1-\lambda^{n} \mu^{n})^2}\,
\frac{1}{1-\lambda^{n-1} \mu^{n}}.
\end{align*}
\end{definition}

\begin{lemma}
\label{lem:dimHcal}
For $(i,j) \in \mathbb N\times \mathbb N$ the following are the
same:
\begin{enumerate}
\item[\rm (i)] the dimension of the $(i,j)$-homogeneous component of
$U^+_q \otimes \mathbb F \lbrack z_1, z_2,\ldots \rbrack$;
\item[\rm (ii)] the coefficient of
$\lambda^i \mu^j$ in $\mathcal H(\lambda, \mu)$.
\end{enumerate}
\end{lemma}
\begin{proof} By Lemmas
\ref{lem:ugrad},
\ref{lem:zgrad}
and Definition \ref{def:calH}.
\end{proof}

\noindent Recall the isomorphism 
$\phi:
U^+_q \otimes \mathbb F \lbrack z_1, z_2,\ldots \rbrack
\to \mathcal U^+_q$ from Lemma
\ref{lem:nphi}
and Theorem
\ref{thm:main1}.

\begin{lemma}
\label{lem:phiresp}
For $ (i,j) \in \mathbb N\times \mathbb N$
the isomorphsim $\phi$ sends the
$ (i,j)$-homogeneous component of 
$U^+_q \otimes \mathbb F \lbrack z_1, z_2,\ldots \rbrack$
to the 
$ (i,j)$-homogeneous component of $\mathcal U^+_q$.
\end{lemma}
\begin{proof}
Use 
Lemmas
\ref{lem:calgrade},
\ref{lem:nphi},
\ref{lem:znn},
\ref{lem:fg}.
\end{proof}

\begin{proposition}
\label{prop:Udim}
For $(i,j) \in \mathbb N\times \mathbb N$ the following are the
same:
\begin{enumerate}
\item[\rm (i)] the dimension of the $(i,j)$-homogeneous component of
$\mathcal U^+_q$;
\item[\rm (ii)] the coefficient of
$\lambda^i \mu^j$ in $\mathcal H(\lambda, \mu)$.
\end{enumerate}
\end{proposition}
\begin{proof} By Lemmas
\ref{lem:dimHcal},
\ref{lem:phiresp}.
\end{proof}

\begin{example}\rm
\label{ex:Ctabledij}
\rm
For $0 \leq i,j\leq 6$ the dimension of the
$(i,j)$-homogeneous component of $\mathcal U^+_q$
is given in the $(i,j)$-entry of the matrix below:
\begin{align*}
	\left(
         \begin{array}{ccccccc}
               1&1&1&1&1&1&1 
	       \\
               1&3&4&4&4&4&4 
               \\
	       1&4&10&13&14&14&14 
              \\
	      1&4&13&27&36&39&40 
             \\
	     1&4&14&36&69&91&101 
            \\
	    1&4&14&39&91&161&213
	    \\
	    1&4&14&40&101&213&361
                  \end{array}
\right)
\end{align*}
\end{example}

\section{A PBW basis for $\mathcal U^+_q$}
In this section we obtain a 
PBW basis for $\mathcal U^+_q$.
First we clarify our terms.

\begin{definition} 
\label{def:PBWcal}
\rm (See  \cite[p.~299]{damiani}.) 
Let $\mathcal A$ denote an algebra.
     A {\it Poincar\'e-Birkhoff-Witt} (or {\it PBW})
   basis for $\mathcal A$ consists of a subset $\Omega \subseteq
    \mathcal A$ and
   a linear order $<$ on $\Omega$, such that  the
   following is a basis for the vector space $\mathcal A$:
    \begin{align*}
    a_1 a_2 \cdots a_n \qquad \quad  n\in \mathbb N, \quad \qquad
    a_1,a_2,\ldots, a_n \in \Omega,
    \qquad \quad  a_1 \leq a_2 \leq \cdots \leq a_n.
    \end{align*}
  We interpret the empty product as the multiplicative identity in 
   $\mathcal A$.
   \end{definition}

\noindent 
Before proceeding, we have a comment about our approach.
Shortly we will apply
\cite[Propositions~6.2,~7.2]{alternating}.
The results in
\cite[Propositions~6.2,~7.2]{alternating} are about
$U^+_q$.
However the proofs of
\cite[Propositions~6.2,~7.2]{alternating} use only
Lemmas
\ref{lem:nrel1},
\ref{lem:nrel2}. Therefore
the results in \cite[Propositions~6.2,~7.2]{alternating} apply to
 $\mathcal U^+_q$ as well as $U^+_q$.
We will apply \cite[Propositions~6.2,~7.2]{alternating} to $\mathcal U^+_q$.

\begin{theorem}
\label{thm:pbwcal}
A PBW basis for $\mathcal U^+_q$ is obtained by its alternating generators
\begin{align*}
\lbrace \mathcal W_{-i} \rbrace_{i \in \mathbb N}, \qquad 
\lbrace \mathcal G_{j+1} \rbrace_{j\in \mathbb N}, \qquad  
\lbrace \mathcal {\tilde G}_{k+1} \rbrace_{k\in \mathbb N}, \qquad  
\lbrace \mathcal W_{\ell+1} \rbrace_{\ell\in \mathbb N}
\end{align*}
in any linear order $<$ that satisfies
\begin{align*}
\mathcal W_{-i} <  \mathcal G_{j+1} < \mathcal{\tilde G}_{k+1}
< \mathcal  W_{\ell+1} \qquad  i,j,k,\ell\in \mathbb N.
\end{align*}
\end{theorem}
\begin{proof} 
Let $\Omega$ denote the 
set of alternating generators for $\mathcal U^+_q$.
Consider the following vectors in 
$\mathcal U^+_q$:
    \begin{align}
    a_1 a_2 \cdots a_n \qquad \quad  n\in \mathbb N, \quad \qquad
    a_1,a_2,\ldots, a_n \in \Omega, 
    \qquad \quad  a_1 \leq a_2 \leq \cdots \leq a_n.
 \label{eq:setcal} 
  \end{align}
We will show that the vectors
 in (\ref{eq:setcal}) form 
a basis for the vector space $\mathcal U^+_q$.
We first show that the vectors
in (\ref{eq:setcal})  span $\mathcal U^+_q$.
To each element of $\Omega$ we assign a weight
as follows. The elements
$\lbrace \mathcal W_{-i} \rbrace_{i \in \mathbb N}$
(resp. $\lbrace \mathcal G_{j+1} \rbrace_{j\in \mathbb N}$)
(resp. $\lbrace \mathcal {\tilde G}_{k+1} \rbrace_{k\in \mathbb N}$)
(resp. $\lbrace \mathcal W_{\ell+1} \rbrace_{\ell\in \mathbb N}$)
get weight 0 (resp. 1) (resp. 2) (resp. 3).
Any two elements of $\Omega$ commute if they have the same weight.
Let $\mathcal S$ denote the subspace of $\mathcal U^+_q$ spanned by
 (\ref{eq:setcal}). 
Note that $\mathcal S$ is spanned by the vectors
    \begin{align}
    a_1 a_2 \cdots a_n \qquad  n\in \mathbb N,  \qquad
    a_1,a_2,\ldots, a_n \in \Omega, 
    \qquad  {\rm wt}(a_1) \leq {\rm wt}(a_2) \leq \cdots \leq 
    {\rm wt}(a_n).
 \label{eq:setcalWT} 
  \end{align}
\noindent The algebra
$\mathcal U^+_q$ is generated by
$\Omega$. Therefore the vector space 
$\mathcal U^+_q$ is spanned by 
\begin{align}
    a_1 a_2 \cdots a_n \qquad \quad  n\in \mathbb N, \quad \qquad
    a_1,a_2,\ldots, a_n \in \Omega.
 \label{eq:setcal1} 
  \end{align}
For any product 
    $a_1 a_2 \cdots a_n$ in
 (\ref{eq:setcal1}), define its {\it defect} to be
$\sum_{i=1}^n (n-i) {\rm wt}(a_i)$.
We assume that 
 $\mathcal S \not=\mathcal U^+_q$ and get a contradiction. 
There exists a product in
 (\ref{eq:setcal1}) that is not contained in $\mathcal S$.
Let $D$ denote the minimum defect of all such products.
Pick a product $a_1a_2\cdots a_n$ in
 (\ref{eq:setcal1}) that is not contained in $\mathcal S$ and
 has defect $D$.
    The product 
    $a_1 a_2 \cdots a_n$ is not listed in 
 (\ref{eq:setcalWT}). Therefore 
there exists an integer $s$ $(2 \leq s \leq n)$
such that
${\rm wt}(a_{s-1}) > {\rm wt}(a_s)$.
 Using 
\cite[Propositions~6.2,~7.2]{alternating} 
we express the product
$a_{s-1}a_s$ as a linear combination of 
products $a'_{s-1} a'_s$ such that $a'_{s-1}, a'_s \in \Omega$
and ${\rm wt}(a'_{s-1}) < {\rm wt}(a'_s)$ and
$
{\rm wt}(a_{s-1})+
{\rm wt}(a_{s})=
{\rm wt}(a'_{s-1})+
{\rm wt}(a'_{s})$.
Replacing 
$a_{s-1}a_s$ by
$a'_{s-1}a'_s$ in $a_1a_2\cdots a_n$
we obtain a product
with defect less than $D$, and hence contained in $\mathcal S$.
We have now expressed $a_1a_2\cdots a_n$ as
a linear combination of
products, each contained in $\mathcal S$.
 Consequently
 $a_1a_2\cdots a_n$ is contained in $\mathcal S$, for a contradiction.
We have shown that the vectors in (\ref{eq:setcal})  span $\mathcal U^+_q$.
We can now easily show that the vectors
in (\ref{eq:setcal}) form a basis for $\mathcal U^+_q$.
Each element in $\Omega$ is 
  homogeneous with respect to the
 $(\mathbb N \times \mathbb N)$-grading of $\mathcal U^+_q$.
So each vector in 
 (\ref{eq:setcal}) is homogeneous with respect to the
 $(\mathbb N \times \mathbb N)$-grading of $\mathcal U^+_q$.
Consequently 
for $(i,j) \in \mathbb N \times \mathbb N$ the
 vectors in 
 (\ref{eq:setcal}) that have degree $(i,j)$
span the $(i,j)$-homogeneous component of
$\mathcal U^+_q$.
The number of such vectors
 is equal to the coefficient of $\lambda^i\mu^j$ in
$\mathcal H(\lambda, \mu)$,
and by Proposition
\ref{prop:Udim}
this coefficient is equal to the dimension
of the $(i,j)$-homogeneous component of $\mathcal U^+_q$.
By these comments and linear algebra, 
the vectors in 
 (\ref{eq:setcal}) that have degree $(i,j)$
form a basis for the $(i,j)$-homogeneous component of
$\mathcal U^+_q$.
We conclude that 
the vectors in 
 (\ref{eq:setcal}) form a basis for
the vector space $\mathcal U^+_q$.
The result follows.
\end{proof}

\section{Directions for  future research}

\noindent 
In this section we give some conjectures concerning
the $q$-Onsager algebra $\mathcal O_q$
and its current algebra $\mathcal A_q$. 
We will use the notation of
Definition \ref{def:polyalg} and
\cite[Definition~3.1]{basnc}.

\begin{conjecture}\rm
There exists an algebra isomorphism
$\mathcal A_q \to \mathcal O_q \otimes 
\mathbb F \lbrack z_1, z_2,\ldots \rbrack$.
\end{conjecture}

\begin{conjecture}\rm
Let $\mathcal Z$ denote the center of 
$\mathcal A_q$. Then the algebra $\mathcal Z$
is isomorphic to 
$\mathbb F \lbrack z_1, z_2,\ldots \rbrack$.
\end{conjecture}

\begin{conjecture}\rm
Let $\langle \mathcal W_0, \mathcal W_1\rangle$
denote the subalgebra of $\mathcal A_q$
generated by $\mathcal W_0, \mathcal W_1$.
Then the algebra 
 $\langle \mathcal W_0, \mathcal W_1\rangle$
is isomorphic to $\mathcal O_q$.
\end{conjecture}

\begin{conjecture}\rm
The multiplication map
\begin{align*}
\langle \mathcal W_0,\mathcal W_1\rangle \otimes \mathcal Z &\to
	       \mathcal A_q 
	       \\
w \otimes z &\mapsto      wz            
\end{align*}
is an algebra isomorphism.
\end{conjecture}

\begin{conjecture}\rm The generators in order
\begin{align*}
\lbrace \mathcal W_{-k} \rbrace_{k \in \mathbb N}, \qquad 
\lbrace \mathcal G_{k+1} \rbrace_{k\in \mathbb N}, \qquad  
\lbrace \mathcal {\tilde G}_{k+1} \rbrace_{k\in \mathbb N}, \qquad  
\lbrace \mathcal W_{k+1} \rbrace_{k\in \mathbb N}
\end{align*}
give a PBW basis for $\mathcal A_q$.
\end{conjecture}

\noindent 
See \cite{basBel} for results that support the
above conjectures.
For more general information on $\mathcal A_q$ 
and $\mathcal O_q$,
see  
\cite{
BK05,
basnc} for 
a mathematical physics point of view,
and 
 \cite{BK, kolb, qSerre,
pospart,
lusztigaut,
pbw,
z2z2z2}
for an algebraic point of view.

\section{Acknowledgment} Some of the research
supporting this paper was carried out by
the author during his visit to the
Institut Denis Poisson (IDP) CNRS
at the University of Tours,
France, during
May 13--24, 2019. 
The author thanks his host
Pascal Baseilhac for many discussions and kind hospitality.
The author also thanks Hiroyuki Yamane at U. Toyama for his explanation
concerning the center of
$U^+_q$.

\section{Appendix: The $Z^\vee_n$ are central in $\mathcal U^+_q$}
\noindent Our goal here is to prove
Lemma
\ref{lem:Zcent}. It will be convenient  to use generating
functions. 
\begin{definition}
\label{def:gf4}
\rm
We define some generating functions in an indeterminate $t$:
\begin{align*}
&\mathcal G(t) = \sum_{n \in \mathbb N} \mathcal G_n t^n,
\qquad \qquad \qquad 
\mathcal {\tilde G}(t) = \sum_{n \in \mathbb N} \mathcal {\tilde G}_n t^n,
\\
&
\mathcal W^-(t) = \sum_{n \in \mathbb N} \mathcal W_{-n} t^n,
\qquad \qquad  
\mathcal W^+(t) = \sum_{n \in \mathbb N} \mathcal W_{n+1} t^n.
\end{align*}
\end{definition}
\noindent 
By (\ref{eq:n3p4vvC}) we have
\begin{align}
\lbrack \mathcal W_0, \mathcal W^-(t) \rbrack = 0, \qquad \qquad
\lbrack \mathcal W_1, \mathcal W^+(t) \rbrack = 0.
\label{eq:com1}
\end{align}

\noindent Next we express the relations
(\ref{eq:n3p1vvC})--(\ref{eq:n3p3vvC})
in terms of generating functions.
\begin{lemma}
\label{lem:gf1} We have
\begin{align*}
&\lbrack  \mathcal W_0, \mathcal W^+(t) \rbrack = 
\lbrack  \mathcal W^-(t), \mathcal W_1 \rbrack = 
(1-q^{-2})t^{-1} \bigl(\mathcal {\tilde G}(t)-\mathcal G(t)\bigr),
\\
&\lbrack  \mathcal W_0, \mathcal G(t) \rbrack_q = 
\lbrack  \mathcal {\tilde G}(t), \mathcal W_0 \rbrack_q = 
(q-q^{-1}) \mathcal  W^-(t),
\\
&\lbrack  \mathcal G(t), \mathcal W_1 \rbrack_q = 
\lbrack  \mathcal W_1, \mathcal {\tilde G}(t) \rbrack_q = 
(q-q^{-1}) \mathcal  W^+(t).
\end{align*}
\end{lemma}

\noindent Next we express the relations
(\ref{eq:n3p4vvC})--(\ref{eq:n3p11vC})
in terms of generating functions. Let $s$ denote an indeterminate
that commutes with $t$.
\begin{lemma}
\label{lem:gf2} We have
\begin{align*}
&\lbrack  \mathcal W^-(s), \mathcal W^-(t) \rbrack = 0, 
\qquad 
\lbrack  \mathcal W^+(s), \mathcal W^+(t) \rbrack = 0,
\\
&\lbrack  \mathcal W^-(s), \mathcal W^+(t) \rbrack 
+
\lbrack  \mathcal W^+(s), \mathcal W^-(t) \rbrack = 0,
\\
&s \lbrack  \mathcal W^-(s), \mathcal G(t) \rbrack 
+
t \lbrack  \mathcal G(s), \mathcal W^-(t) \rbrack = 0,
\\
&s \lbrack  \mathcal W^-(s), \mathcal {\tilde G}(t) \rbrack 
+
t \lbrack  \mathcal{\tilde G}(s), \mathcal W^-(t) \rbrack = 0,
\\
&s \lbrack  \mathcal W^+(s), \mathcal G(t) \rbrack
+
t \lbrack  \mathcal G(s), \mathcal W^+(t) \rbrack = 0,
\\
&s \lbrack  \mathcal W^+(s), \mathcal {\tilde G}(t) \rbrack
+
t \lbrack  \mathcal{\tilde G}(s), \mathcal W^+(t) \rbrack = 0,
\\
&\lbrack  \mathcal G(s), \mathcal G(t) \rbrack = 0, 
\qquad 
\lbrack  \mathcal {\tilde G}(s), \mathcal {\tilde G}(t) \rbrack = 0,
\\
&\lbrack  \mathcal {\tilde G}(s), \mathcal G(t) \rbrack +
\lbrack  \mathcal G(s), \mathcal {\tilde G}(t) \rbrack = 0.
\end{align*}
\end{lemma}

\noindent Next we express the relations
(\ref{eq:calngg1})--(\ref{eq:calngg6})
in terms of generating functions.
\begin{lemma}
\label{lem:gf3}
We have
\begin{align*}
&\lbrack \mathcal W^-(s), \mathcal G(t)\rbrack_q = 
\lbrack \mathcal W^-(t), \mathcal G(s)\rbrack_q,
\qquad \quad
\lbrack \mathcal G(s), \mathcal W^+(t)\rbrack_q = 
\lbrack \mathcal G(t), \mathcal W^+(s)\rbrack_q,
\\
&
\lbrack \mathcal {\tilde G}(s), \mathcal W^-(t)\rbrack_q = 
\lbrack \mathcal {\tilde G}(t), \mathcal W^-(s)\rbrack_q,
\qquad \quad 
\lbrack \mathcal W^+(s), \mathcal {\tilde G}(t)\rbrack_q = 
\lbrack \mathcal W^+(t), \mathcal {\tilde G}(s)\rbrack_q,
\\
&t^{-1}\lbrack \mathcal G(s), \mathcal {\tilde G}(t)\rbrack -
s^{-1}\lbrack \mathcal G(t), \mathcal {\tilde G}(s)\rbrack =
q\lbrack \mathcal W^-(t), \mathcal W^+(s)\rbrack_q-
q\lbrack \mathcal W^-(s), \mathcal W^+(t)\rbrack_q,
\\
&t^{-1}\lbrack \mathcal  {\tilde G}(s), \mathcal  G(t)\rbrack -
s^{-1}\lbrack \mathcal  {\tilde G}(t), \mathcal G(s)\rbrack =
q \lbrack \mathcal  W^+(t), \mathcal W^-(s)\rbrack_q-
q\lbrack \mathcal W^+(s), \mathcal W^-(t)\rbrack_q,
\\
&\lbrack \mathcal G(s), \mathcal {\tilde G}(t)\rbrack_q -
\lbrack \mathcal G(t), \mathcal {\tilde G}(s)\rbrack_q =
qt\lbrack \mathcal W^-(t), \mathcal W^+(s)\rbrack-
qs\lbrack \mathcal W^-(s), \mathcal W^+(t)\rbrack,
\\
&\lbrack \mathcal {\tilde G}(s),  \mathcal G(t)\rbrack_q -
\lbrack \mathcal {\tilde G}(t),  \mathcal G(s)\rbrack_q =
qt \lbrack \mathcal W^+(t), \mathcal W^-(s)\rbrack-
qs\lbrack \mathcal W^+(s), \mathcal W^-(t)\rbrack.
\end{align*}
\end{lemma}

\noindent So far in this section
we displayed many relations involving the generating functions
from Definition
\ref{def:gf4}.
In the next two lemmas we express these relations
in a more convenient form.

\begin{lemma}
\label{lem:redrel1}
We have
\begin{align*}
\mathcal W^-(t) \mathcal W_0 
&= 
\mathcal W_0 \mathcal W^-(t),
\\
\mathcal W^+(t) \mathcal W_0 
&= 
\mathcal W_0 \mathcal W^+(t)+(1-q^{-2})t^{-1} 
\bigl(\mathcal G(t)-\mathcal {\tilde G}(t)\bigr),
\\
\mathcal G(t) \mathcal W_0 
&= 
q^2 \mathcal W_0 \mathcal G(t)+(1-q^{2})
\mathcal W^-(t),
\\
\mathcal {\tilde G}(t) \mathcal W_0 
&= 
q^{-2} \mathcal W_0 \mathcal {\tilde G}(t)
+(1-q^{-2})
\mathcal W^-(t)
\end{align*}
\noindent and
\begin{align*}
\mathcal W_1 \mathcal W^+(t)  
&= 
\mathcal W^+(t) \mathcal W_1,
\\
\mathcal W_1 \mathcal W^-(t) 
&= 
\mathcal W^-(t) \mathcal W_1+(1-q^{-2})t^{-1} 
\bigl(\mathcal G(t)-\mathcal {\tilde G}(t)\bigr),
\\
\mathcal W_1 \mathcal G(t) 
&= 
q^2 \mathcal G(t)  \mathcal W_1 +(1-q^{2})
\mathcal W^+(t),
\\
\mathcal W_1 \mathcal {\tilde G}(t) 
&= 
q^{-2} \mathcal {\tilde G}(t)\mathcal W_1 
+(1-q^{-2})
\mathcal W^+(t).
\end{align*}
\end{lemma}
\begin{proof} These are reformulations of
(\ref{eq:com1}) and Lemma
\ref{lem:gf1}.
\end{proof}

\begin{lemma}
\label{lem:redrel2}
We have
\begin{align*}
\mathcal G(s) \mathcal W^-(t) &= 
q \frac{(qs-q^{-1}t)\mathcal W^-(t) \mathcal G(s) - (q-q^{-1})s 
\mathcal W^-(s) \mathcal G(t)}{s-t},
\\
\mathcal {\tilde G}(s) \mathcal W^-(t) &= 
q^{-1} \frac{(q^{-1}s-qt)\mathcal W^-(t) \mathcal {\tilde G}(s) + (q-q^{-1})s 
\mathcal W^-(s) \mathcal {\tilde G}(t)}{s-t},
\\
\mathcal W^+(s) \mathcal G(t) &= 
q \frac{(q^{-1}s-qt)\mathcal G(t) \mathcal W^+(s) + (q-q^{-1})t 
\mathcal G(s) \mathcal W^+(t)}{s-t},
\\
\mathcal W^+(s) \mathcal {\tilde G}(t) &= 
q^{-1} \frac{(qs-q^{-1}t)\mathcal {\tilde G}(t) \mathcal W^+(s) - (q-q^{-1})t 
\mathcal {\tilde G}(s) \mathcal W^+(t)}{s-t}
\end{align*}
\noindent 
and
\begin{align*}
\mathcal W^+(s) \mathcal W^-(t) &= 
\mathcal W^-(t) \mathcal W^+(s) +
(1-q^{-2})\frac{\mathcal G(s) \mathcal {\tilde G}(t)-\mathcal G(t)
\mathcal {\tilde G}(s)}{s-t},
\\
\mathcal {\tilde G}(s) \mathcal G(t) &= 
\mathcal G(t) \mathcal {\tilde G}(s) +
(1-q^{2})st \frac{\mathcal W^-(t) \mathcal W^{+}(s)-\mathcal W^-(s)
\mathcal W^+(t)}{s-t}.
\end{align*}
\end{lemma}
\begin{proof} To obtain the first equation of the lemma statement,
consider the relations
\begin{align*}
&
s \lbrack \mathcal W^-(s), \mathcal G(t)\rbrack +
t \lbrack \mathcal G(s), \mathcal W^-(t) \rbrack = 0, 
\\
&
 \lbrack \mathcal W^-(s), \mathcal G(t)\rbrack_q = 
 \lbrack \mathcal W^-(t), \mathcal G(s) \rbrack_q
\end{align*}
from Lemmas
\ref{lem:gf2},
\ref{lem:gf3}.
These relations give a system of linear equations in two
unknowns
$\mathcal G(s) \mathcal W^-(t)$,
$\mathcal G(t) \mathcal W^-(s)$. Solve this system using
linear algebra to obtain the first equation in the lemma statement.
The next three equations in the lemma statement are similarly obtained.
To obtain the last two equations in the lemma statement, consider
the last four relations in
Lemma \ref{lem:gf3}. These relations give a system of linear
equations in four unknowns
\begin{align*}
&\mathcal W^+(s)\mathcal W^-(t),
\qquad \quad
\mathcal W^+(t)\mathcal W^-(s),
\qquad \quad
\mathcal {\tilde G}(s)\mathcal G(t),
\qquad \quad
\mathcal {\tilde G}(t)\mathcal  G(s).
\end{align*}
Solve this system using linear algebra to obtain 
the last two equations in the lemma statement.
\end{proof}

\noindent The relations in Lemmas
\ref{lem:redrel1},
\ref{lem:redrel2} will be called {\it reduction rules}.

\begin{definition}
\label{def:zgen}
\rm
Define the generating function
\begin{align*}
Z^\vee(t) = \sum_{n\in \mathbb N} Z^\vee_n t^n.
\end{align*}
\end{definition}

\begin{lemma} 
\label{lem:ZvvGF}
We have
\begin{align*}
Z^\vee(t) = \mathcal G(q^{-1}t) \mathcal{\tilde G}(qt) -qt
\mathcal W^{-} (q^{-1}t)\mathcal W^+(qt).
\end{align*}
\end{lemma}
\begin{proof} 
By Definitions
\ref{def:Zvee},
\ref{def:gf4},
\ref{def:zgen}.
\end{proof}

\begin{lemma} 
\label{lem:hint1}
We have
\begin{align*}
&\lbrack \mathcal W_0, Z^\vee(t)\rbrack=0,
\qquad \qquad 
\lbrack \mathcal W_1, Z^\vee(t)\rbrack=0.
\end{align*}
\end{lemma}
\begin{proof} To verify each equation,
eliminate $Z^\vee(t)$ using
Lemma \ref{lem:ZvvGF}, and evaluate the result
using the reduction rules.
\end{proof}

\begin{lemma}
\label{lem:hint2}
We have
\begin{align*}
&\lbrack \mathcal G(s), Z^\vee(t)\rbrack=0,
\qquad \qquad
\lbrack \mathcal {\tilde G}(s), Z^\vee(t)\rbrack=0,
\\
&\lbrack \mathcal W^-(s), Z^\vee(t)\rbrack=0,
\qquad \qquad
\lbrack \mathcal W^+(s), Z^\vee(t)\rbrack=0.
\end{align*}
\end{lemma}
\begin{proof} To verify
$\lbrack \mathcal G(s), Z^\vee(t)\rbrack=0$,
eliminate $Z^\vee(t)$ using Lemma
\ref{lem:ZvvGF}, and evaluate the result
using the reduction rules.
To obtain the remaining three equations in
the lemma statement, 
use Lemmas
\ref{lem:gf1}, 
\ref{lem:hint1}.
\end{proof}

\noindent We can now easily prove Lemma
\ref{lem:Zcent}. 
Let $n\in \mathbb N$ be given.
By Lemma
\ref{lem:hint2},
$Z^\vee_n$ commutes with every alternating generator
of $\mathcal U^+_q$, and hence everything in
$\mathcal U^+_q$.
In other words, $Z^\vee_n$ is central in
$\mathcal U^+_q$.

\bigskip

\noindent Paul Terwilliger \hfil\break
\noindent Department of Mathematics \hfil\break
\noindent University of Wisconsin \hfil\break
\noindent 480 Lincoln Drive \hfil\break
\noindent Madison, WI 53706-1388 USA \hfil\break
\noindent email: {\tt terwilli@math.wisc.edu }\hfil\break


\begin{thebibliography}{10}

\bibitem{bas2}
P.~Baseilhac.
\newblock An integrable structure related with tridiagonal
algebras.
\newblock {\em
Nuclear Phys. B}
 705
 (2005)
 605--619;
 {\tt arXiv:math-ph/0408025}.



\bibitem{bas1}
P.~Baseilhac.
\newblock Deformed {D}olan-{G}rady relations in quantum integrable
models.
\newblock {\em
Nuclear Phys. B}
 709
 (2005)
 491--521;
 {\tt arXiv:hep-th/0404149}.





\bibitem{baspp}
P.~Baseilhac.
\newblock The positive part of
$U_q(\widehat{\mathfrak{sl}}_2)$ and tensor product
representations;
\newblock {\em
preprint}. 





\bibitem{basXXZ}
P.~Baseilhac and S.~Belliard.
\newblock
The half-infinite XXZ chain in Onsager's approach.
\newblock
{\em
Nuclear Phys. B} 873 (2013) 550--584;
 {\tt arXiv:1211.6304}.



\bibitem{basBel}
P.~Baseilhac and S.~Belliard.
\newblock An attractive basis for the $q$-Onsager algebra;
      {\tt arXiv:1704.02950}.


\bibitem{BK05}
P.~Baseilhac and K.~Koizumi.
\newblock A new (in)finite dimensional algebra for
quantum integrable models.
\newblock {\em 
 Nuclear Phys. B}  720  (2005) 325--347;
   {\tt arXiv:math-ph/0503036}.



\bibitem{bas4}
 P.~Baseilhac and K.~Koizumi.
   \newblock
    A deformed analogue of Onsager's symmetry
      in the
       $XXZ$ open spin chain.
       \newblock {\em
        J. Stat. Mech. Theory Exp.} 
 2005,  no. 10, P10005, 15 pp. (electronic);
	  {\tt arXiv:hep-th/0507053}.

  \bibitem{basKoi}
    P.~Baseilhac and K.~Koizumi.
       \newblock
    Exact spectrum of the $XXZ$ open spin chain from the
          $q$-Onsager algebra representation theory.
          \newblock {\em  J. Stat. Mech. Theory Exp.}
       2007,  no. 9, P09006, 27 pp. (electronic);
      {\tt arXiv:hep-th/0703106}.


\bibitem{BK}
P.~Baseilhac and S.~Kolb.
\newblock Braid group action and root vectors for the $q$-Onsager algebra;
\newblock 
   {\tt arXiv:1706.08747}.

\bibitem{basnc}
 P.~Baseilhac and K.~Shigechi.
  \newblock
   A new current algebra and the reflection equation.
    \newblock{\em
     Lett. Math. Phys. }
       92
        (2010)   47--65;
{\tt arXiv:0906.1482v2}.


\bibitem{damiani}
I.~Damiani.
\newblock
A basis of type Poincare-Birkoff-Witt for the quantum
algebra of $\widehat{\mathfrak{sl}}_2$.
\newblock {\em
J. Algebra} 161 (1993) 291--310.








\bibitem{kolb}
S.~Kolb.
\newblock
Quantum symmetric Kac-Moody pairs.
\newblock {\em
Adv. Math.} 267 (2014) 395-469;
{\tt arXiv:1207.6036}.



\bibitem{lusztig}
G.~Lusztig.
\newblock
{\em
Introduction to quantum groups}.
\newblock
Progress in Mathematics, 110. Birkhauser, Boston, 1993.



\bibitem{rosso1}
M.~Rosso.
\newblock
Groupes quantiques et alg{\`e}bres de battage quantiques.
\newblock{\em
C.~R. Acad. Sci. Paris} 320 (1995) 145--148.

\bibitem{rosso}
M.~Rosso.
\newblock
Quantum groups and quantum shuffles.
\newblock{\em
Invent. Math} 133 (1998) 399--416.




 \bibitem{qSerre}
  P.~Terwilliger.
    \newblock Two relations that generalize the $q$-Serre relations and the
      Dolan-Grady relations. In
        \newblock {\em  Physics and
	  Combinatorics 1999 (Nagoya)}, 377--398, World Scientific Publishing,
	     River Edge, NJ, 2001;
	     {\tt arXiv:math.QA/0307016}.




\bibitem{pospart}
P.~Terwilliger.
\newblock
The $q$-Onsager algebra and the positive part of
$U_q(\widehat{\mathfrak{sl}}_2)$.
\newblock{\em
 Linear Algebra Appl.}
 521 (2017) 19--56;
{\tt arXiv:1506.08666}.

\bibitem{lusztigaut}
P.~Terwilliger.
\newblock
The Lusztig automorphism of the $q$-Onsager algebra.
\newblock{\em
J. Algebra.} 506 (2018) 56--75;
{\tt arXiv:1706.05546}.

\bibitem{pbw}
P.~Terwilliger.
\newblock
The $q$-Onsager algebra and the universal Askey-Wilson algebra.
\newblock{\em SIGMA Symmetry Integrability Geom. Methods Appl.}
14 (2018) Paper No. 044, 18 pp.






\bibitem{z2z2z2}
P.~Terwilliger.
\newblock
An action of the free product
$\mathbb Z_2 \star \mathbb Z_2 \star \mathbb Z_2$ on the
$q$-Onsager algebra and its current algebra.
\newblock {\em
Nuclear Phys. B}
936 
 (2018)
 306--319;
{\tt arXiv:1808.09901}.



\bibitem{alternating}
P.~Terwilliger.
\newblock
The alternating PBW basis for the positive part
of $U_q(\widehat{\mathfrak{sl}}_2)$;
\newblock
{\tt arXiv:1902.00721}.

\bibitem{yamane}
H.~Yamane.
\newblock Personal communication.


 \end{thebibliography}
\end{document}